\documentclass{article}

\usepackage
{
    graphicx,
    amssymb,
    amsmath,
    amsthm,
    dsfont, 
    xcolor,
    bm,
    bbm,
    enumerate,
    authblk,
    algorithm,
    algpseudocode,
    colortbl,
}

\usepackage[colorlinks,
            pdffitwindow=false,
            plainpages=false,
            pdfpagelabels=true,
            pdfpagemode=UseOutlines,
            pdfpagelayout=SinglePage,
            bookmarks=false,
            colorlinks=true,
            hyperfootnotes=false,
            linkcolor=blue,
            hypertexnames=false,
            citecolor=green!50!black]{hyperref}

\usepackage[hmargin=2cm,vmargin=2.5cm]{geometry}
\usepackage[bf]{caption}
\usepackage{subcaption}

\DeclareMathAlphabet{\mathpzc}{OT1}{pzc}{m}{it}

\makeatletter
\newlength{\multiindent}
\algnewcommand{\LineCommentCont}[1]{\Statex \hskip\ALG@thistlm%
  \parbox[t]{\dimexpr\linewidth-\ALG@thistlm}
{\leftskip=\multiindent
  \hangindent=\multiindent
  \hangafter=1
  \strut\makebox[\multiindent][c]{}#1\strut}
  } 
\makeatother

\newcommand{\R}{\mathbb{R}}                                     
\newcommand{\pd}[2]{\frac{\partial#1}{\partial#2}}              

\newcommand\xqed[1]{\leavevmode\unskip\penalty9999 \hbox{}\nobreak\hfill \quad\hbox{#1}}
\newcommand{\exampleSymbol}{\xqed{$\triangle$}}

\DeclareMathOperator{\diag}{diag}

\newtheorem{theorem}{Theorem}[section]
\newtheorem*{theorem*}{Theorem}

\newtheorem{definition}[theorem]{Definition}
\theoremstyle{definition}
\newtheorem{example}[theorem]{Example}
\newtheorem{remark}[theorem]{Remark}

\makeatletter
\renewcommand*\env@matrix[1][*\c@MaxMatrixCols c]{%
  \hskip -\arraycolsep
  \let\@ifnextchar\new@ifnextchar
  \array{#1}}
\makeatother

\title{Nearest-Neighbor Interaction Systems in the Tensor-Train Format}
\author[1]{Patrick Gel\ss}
\author[1]{Stefan Klus}
\author[1]{Sebastian Matera}
\author[1,2]{Christof Sch\"utte}
\affil[1]{Department of Mathematics and Computer Science, Freie Universit\"at Berlin, Germany}
\affil[2]{Zuse Institute Berlin, Germany}
\date{}

\begin{document}
\maketitle

\begin{abstract}
Low-rank tensor approximation approaches have become an important tool in the scientific computing community. The aim is to enable the simulation and analysis of high-dimensional problems which cannot be solved using conventional methods anymore due to the so-called curse of dimensionality. This requires techniques to handle linear operators defined on extremely large state spaces and to solve the resulting systems of linear equations or eigenvalue problems. In this paper, we present a systematic tensor-train decomposition for nearest-neighbor interaction systems which is applicable to a host of different problems. With the aid of this decomposition, it is possible to reduce the memory consumption as well as the computational costs significantly. Furthermore, it can be shown that in some cases the rank of the tensor decomposition does not depend on the network size. The format is thus feasible even for high-dimensional systems. We will illustrate the results with several guiding examples such as the Ising model, a system of coupled oscillators, and a CO oxidation model.
\end{abstract}

\section{Introduction}
\label{sec:Introduction}

Over the last years, the interest in low-rank tensor decompositions has been growing rapidly and several different tensor formats such as the canonical format, the Tucker format, and the TT format have been proposed. It was shown that tensor-based methods can be successfully applied to many different application areas, e.g.~quantum physics~\cite{WHITE, MEYER}, chemical reaction dynamics~\cite{JAHNKE, KAZEEV, DOLGOV01, KS15, GELSS2016}, stochastic queuing problems~\cite{BUCHHOLZ, KRESSNER}, machine learning~\cite{BEYLKIN2009, NOVIKOV2015, COHEN2015}, and high-dimensional data analysis~\cite{KLUS}. The applications typically require solving systems of linear equations, eigenvalue problems, ordinary differential equations, partial differential equations, or completion problems. One of the most promising tensor formats for these problems is the so-called \emph{tensor-train format} (TT format) \cite{OSELEDETS01, OSELEDETS02, OSELEDETS03}, a special case of the \emph{hierarchical Tucker Format} \cite{HK09, Gra10, ARNOLD, LUBICH01}. In this paper, we will consider in particular high-dimensional interaction networks described by a \emph{Markovian master equation} (MME). The goal is to derive systematic TT decompositions of high-dimensional tensors for interaction networks that are only based on nearest-neighbor interactions. In this way, we want to simplify the construction of tensor-train representations, which is one of the most challenging tasks in the tensor-based simulation of interaction networks. The resulting TT decomposition can be easily scaled to describe different state space sizes, e.g.~number of reaction sites or number of species. The complexity of a tensor train is determined by the TT ranks. Not only the memory consumption of the tensor operators is affected by the ranks, but also the costs of standard operations such as the calculation of norms and the runtimes of tensor-based solvers.

Many applications require solving high-dimensional systems of linear equations of the form $\mathbf{A} \cdot \mathbf{T} = \mathbf{U}$, where $\mathbf{A}$ is a linear TT operator and $\mathbf{T}$ and $\mathbf{U}$ are tensors in the TT format. The efficiency of algorithms proposed so far for solving such systems such as ALS, MALS, or AMEn~\cite{HOLTZ01, DOLGOV03, DOLGOV01} depends strongly on the TT ranks of the operator. As a result, it is vitally important to be able to generate low-rank representations of $ \mathbf{A} $. This can be achieved by truncating the operator, neglecting singular values that are smaller than a given threshold $ \varepsilon $, or by exploiting inherent properties of the problem. Our goal is to derive a low-rank decomposition which represents the operator $ \mathbf{A} $ associated with nearest-neighbor interaction networks exactly, without requiring truncation.

Finding a TT decomposition of a given tensor is in general cumbersome, in particular if the number of cells is not determined a priori. In \cite{GELSS2016}, we derived a systematic decomposition for a specific reaction system and it turned out that the underlying idea can be generalized to describe a larger class of interaction systems. The only assumption we make is that the system comprises only nearest-neighbor interactions. The number and types of the cells as well as the interactions between these cells may differ. Moreover, systems with a cyclic network structure can be represented using the proposed TT decomposition. One of the main advantages of the presented decomposition is that the TT ranks of homogeneous systems do not depend on the number of cells of the network. 

The paper is organized as follows: In Section~\ref{sec: theory}, we give a brief overview of the TT format and define a special core notation. Furthermore, nearest-neighbor interaction systems defined on a set of cells and Markovian master equations are introduced. In Section~\ref{sec: SLIMTT}, a specific TT decomposition is derived exploiting properties of nearest-neighbor interaction systems. In Section~\ref{sec: SLIM Markov}, we use this TT decomposition for Markovian master equations in the TT format and present numerical results. Section \ref{sec: concl} concludes with a brief summary and a future outlook.

\section{Theoretical Background}
\label{sec: theory}

In this section, we will introduce the concept of tensors and different tensor decompositions, namely the canonical format and the TT format. Furthermore, we will define interaction systems that are based only on nearest-neighbor couplings. We will distinguish between homogeneous and heterogeneous systems as well as between cyclic and non-cyclic systems.

\subsection{Tensor Formats}\label{subsec: tensor formats}

Tensors, in our sense, are simply multidimensional generalizations of vectors and matrices. A tensor in full format is given by a multidimensional array of the form $ \mathbf{T} \in \R^{n_1 \times \dots \times n_d} $ and a linear operator by $ \mathbf{A} \in \R^{(n_1 \times n_1) \otimes \dots \otimes (n_d \times n_d)} $. The entries of a tensor are indexed by $ \mathbf{T}_{x_1, \dots, x_d} $ and the entries of operators by $ \mathbf{A}_{x_1, y_1 \dots, x_d, y_d} $. In order to mitigate the curse of dimensionality, that is, the exponential growth of the memory consumption in $ d $, various tensor formats have been proposed over the last years. Here, we will focus on the TT format. The common basis of various tensor formats is the tensor product which enables the decomposition of high-dimensional tensors into several smaller tensors. 

\begin{definition} \label{def:outer product}
    The \emph{tensor product} of two tensors $\mathbf{T} \in \mathbb{R}^{m_1 \times \dots \times m_d}$ and $\mathbf{U} \in \mathbb{R}^{n_1 \times \dots \times n_e}$ defines a tensor $\mathbf{T} \otimes \mathbf{U} \in \R^{(m_1 \times \dots \times m_d) \times (n_1 \times \dots \times n_d)}$ with
    \begin{equation*}
        \left(\mathbf{T} \otimes \mathbf{U} \right)_{x_1, \dots, x_d, y_1, \dots, y_e } = \mathbf{T}_{x_1, \dots, x_d} \cdot \mathbf{U}_{y_1, \dots, y_e},
    \end{equation*}
    where $1 \leq x_k \leq m_k$ for $k = 1, \dots, d$ and $1 \leq y_k \leq n_k$ for $k = 1, \dots, e$.
\end{definition}

The tensor product is a bilinear map. That is, if we fix one of the tensors we obtain a linear map on the space where the other tensor lives (see Appendix~\ref{app: Tensor Notation and Properties}). The initial concept of tensor decompositions was introduced in 1927 by Hitchcock \cite{HITCHCOCK}, who presented the idea of expressing a tensor as the sum of a finite number of \emph{rank-one tensors} (or \emph{elementary tensors}).

\begin{definition} 
A tensor $\mathbf{T} \in \mathbb{R}^{n_1 \times \dots \times n_d}$ is said to be in the \emph{canonical format} if
\begin{equation} \label{cf_simple_tensor}
    \mathbf{T} =  \sum_{k=1}^r \left( \mathbf{T}^{(1)} \right)_{k,:} \otimes \dots \otimes \left( \mathbf{T}^{(d)} \right)_{k,:},
\end{equation}
with \emph{cores} $\mathbf{T}^{(i)} \in \mathbb{R}^{r \times n_i}$ for $i=1, \dots, d$. The parameter $r$ is called the \emph{rank} of the decomposition.
\end{definition}

\begin{remark}\label{re: cyclic permutation}
Given a canonical tensor $\mathbf{T}$ as defined in \eqref{cf_simple_tensor}, a cyclic permutation of the cores yields a tensor whose indices are permuted correspondingly. That is, if we define
\begin{equation*}
    \tilde{\mathbf{T}} = \sum_{k=1}^r \left( \mathbf{T}^{(m)} \right)_{k,:} \otimes \dots \otimes \left( \mathbf{T}^{(d)} \right)_{k,:} \otimes \left( \mathbf{T}^{(1)} \right)_{k,:} \otimes \dots \otimes \left( \mathbf{T}^{(m-1)} \right)_{k,:},
\end{equation*}
with $1 \leq m \leq d$, we obtain
\begin{equation*}
    \tilde{\mathbf{T}}_{x_m, \dots, x_d, x_1, \dots, x_{m-1}} = \mathbf{T}_{x_1, \dots, x_{m-1}, x_m, \dots,  x_d}.
\end{equation*}
An analogous statement even holds for arbitrary permutations of the cores. However, we only consider cyclic permutations in order to derive tensor representations for the considered interaction systems.
\end{remark}

In fact, any tensor can be represented by a linear combination of elementary tensors as in \eqref{cf_simple_tensor}. However, the number of required rank-one tensors plays an important role. Although the canonical format can theoretically be used for low-parametric decompositions of high-dimensional tensors~\cite{CARROLL1970, SIDIROPOULOS2000, LATHAUWER2007}, it has a crucial drawback: Since canonical tensors with bounded rank $r$ do not form a manifold, optimization problems can be ill-posed~\cite{SILVA}, with the result that the best approximation may not even exist. For more information about canonical tensors, we refer to~\cite{KOLDA}.

We will use the canonical format for the derivation of TT decompositions for systems based on nearest-neighbor interactions. For the actual computations, however, we will rely on the TT format. The TT format, which was developed by Oseledets and Tyrtyshnikov in 2009, see~\cite{OSELEDETS01, OSELEDETS02}, is a promising candidate for numerical computations due to its stability from an algorithmic point of view and reasonable computational costs.

\begin{definition}
A tensor $\mathbf{T} \in \mathbb{R}^{n_1 \times \dots \times n_d}$ is said to be in the \emph{TT format} if
\begin{equation*} 
    \mathbf{T} = \sum_{k_0=1}^{r_0} \cdots  \sum_{k_d=1}^{r_d} \left( \mathbf{T}^{(1)} \right)_{k_0,:,k_1} \otimes \dots \otimes \left( \mathbf{T}^{(d)} \right)_{k_{d-1},:,k_d},
\end{equation*}
where the $\mathbf{T}^{(i)} \in \mathbb{R}^{r_{i-1} \times n_i \times r_i}$, $i=1, \dots, d$, are called \emph{TT cores} and the numbers $r_i$ \emph{TT ranks} of the tensor. Here, $ r_0 = r_d = 1 $.
\end{definition}

It is important to note that the TT ranks determine the numerical complexity. The lower the ranks, the lower the memory consumption and the computational costs. However, high ranks might be required to represent the state of the network or the solution of a system of linear equations accurately. Nevertheless, it is advantageous to use the TT format since the TT ranks often exhibit a better behavior than the ranks of canonical decompositions, especially for increasing system sizes. In general, the TT ranks are bounded by the canonical rank when expressing the same tensor in both formats, see \cite{OSELEDETS2010}.

The TT decomposition presented in this paper can also be used to express linear operators $ \mathbf{A} $ in the TT format. We assume that these operators are generalizations of square matrices, i.e.~$\mathbf{A} \in \R^{(n_1 \times n_1) \times  \dots \times (n_d \times n_d)}$.

\begin{definition}
A linear operator $\mathbf{A} \in \R^{(n_1 \times n_1) \times \dots \times (n_d \times n_d)}$ is said to be in the \emph{TT format} if
\begin{equation*} 
    \mathbf{A} = \sum_{k_0=1}^{r_0} \cdots  \sum_{k_d=1}^{r_d} \left( \mathbf{A}^{(1)} \right)_{k_0,:,:,k_1} \otimes \dots \otimes \left( \mathbf{A}^{(d)} \right)_{k_{d-1},:,:,k_d},
\end{equation*}
with \emph{TT cores} $\mathbf{A}^{(i)} \in \mathbb{R}^{r_{i-1} \times n_i \times n_i \times r_i}$ for $i=1, \dots, d$. Here, we require again that $ r_0 = r_d = 1 $.
\end{definition}

Figure~\ref{fig: tensor trains} shows a graphical representation of a tensor train $\mathbf{T} \in \R^{n_1 \times \dots \times n_d}$ as well as a TT operator $\mathbf{A} \in \R^{(n_1 \times n_1) \times \dots \times (n_d \times n_d)}$. A core is depicted as a circle with different arms indicating the modes of the tensor and the rank indices. Due to the fact that $r_0$ and $r_d$ are equal to 1, we regard the first and the last TT core as matrices. Analogously, the first and the last core of $\mathbf{A}$ are interpreted as tensors of order 3.

\begin{figure}[htb]
    \centering
    \begin{subfigure}[b]{0.4\textwidth}
        \centering
        \caption{}
        \vspace*{6.5ex}
        \includegraphics[width=0.9\textwidth]{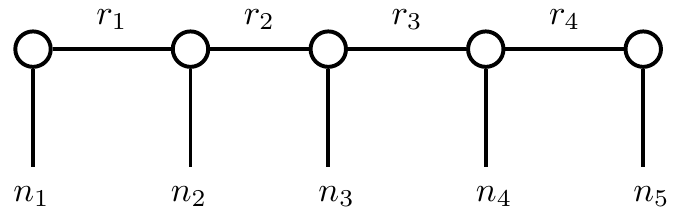}
    \end{subfigure}
    \hspace*{2em}
    \begin{subfigure}[b]{0.4\textwidth}
        \centering
        \caption{}
        \includegraphics[width=0.9\textwidth]{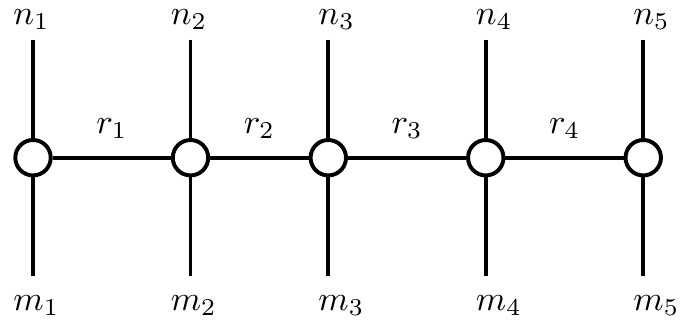}
    \end{subfigure}
    \caption{Graphical representation of tensors: (a) Tensor of order 5 in the TT format, the first and the last core are matrices, the other cores are tensors of order 3. (b) Linear operator of order 5 in the TT format, the first and the last core are tensors of order 3, the other cores are tensors of order 4.}
    \label{fig: tensor trains}
\end{figure}

As for the canonical format, the storage consumption of tensors in the TT format depends only linearly on the number of dimensions. For problems with a certain structure, one can indeed bound the ranks in order to achieve a linear scaling, see e.g~\cite{GELSS2016}. One of the main advantages of the TT format over the canonical format is its stability from an algorithmic point of view \cite{HOLTZ01}. An important property of the TT format is the ensured existence of a best approximation with bounded TT ranks \cite{HOLTZ01, FALCO}. With the aid of the TT format, we are able to avoid the curse of dimensionality -- provided the modes and ranks are reasonably small -- and we can compute quasi-optimal approximations of high-dimensional tensors. Thus, in order to speed up calculations and to be able to solve even high-dimensional problems, low-rank TT representations of linear operators are of utmost importance.

For the sake of comprehensibility, we represent the TT cores as 2-dimensional arrays containing matrices as elements, cf.~\cite{KK12}. For a given tensor-train operator $\mathbf{A} \in \R^{ (n_1 \times n_1 ) \times \dots \times (n_d \times n_d)}$ with cores $\mathbf{A}^{(i)} \in \R^{r_{i-1} \times n_i \times n_i \times r_i}$, $i = 1, \dots, d$, each core is written as
\begin{equation} \label{eq: core notation - single core - operator}
    \left[ \mathbf{A}^{(i)} \right] =
    \begin{bmatrix}
        & \mathbf{A}^{(i)}_{1,:,:,1} & \cdots & \mathbf{A}^{(i)}_{1,:,:,r_i} & \\
        & & & & \\
        & \vdots & \ddots & \vdots & \\
        & & & & \\
        & \mathbf{A}^{(i)}_{r_{i-1},:,:,1} & \cdots & \mathbf{A}^{(i)}_{r_{i-1},:,:,r_i} &
    \end{bmatrix}.
\end{equation}
We then use the following notation for a tensor train $\mathbf{A} \in \R^{(n_1 \times n_1 ) \times \dots \times (n_d \times n_d)}$:
\begin{equation*}
\begin{split}
    \mathbf{A} & = \left[ \mathbf{A}^{(1)}\right] \otimes \left[ \mathbf{A}^{(2)}\right] \otimes \dots \otimes \left[ \mathbf{A}^{(d-1)}\right] \otimes \left[ \mathbf{A}^{(d)}\right] \\[1.5ex]
    &= 
    \begin{bmatrix}
        \mathbf{A}^{(1)}_{1,:,:,1} & \cdots & \mathbf{A}^{(1)}_{1,:,:,r_1}
    \end{bmatrix}
    \otimes 
    \begin{bmatrix}
        \mathbf{A}^{(2)}_{1,:,:,1}   & \cdots & \mathbf{A}^{(2)}_{1,:,:,r_2}  \\
        \vdots                       & \ddots & \vdots                        \\
        \mathbf{A}^{(2)}_{r_1,:,:,1} & \cdots & \mathbf{A}^{(2)}_{r_1,:,:,r_2}
    \end{bmatrix}
    \otimes \cdots \\[1.5ex]
    & \qquad \cdots \otimes
    \begin{bmatrix}
        \mathbf{A}^{(d-1)}_{1,:,:,1}       & \cdots & \mathbf{A}^{(d-1)}_{1,:,:,r_{d-1}}      \\
        \vdots                             & \ddots & \vdots                                  \\
        \mathbf{A}^{(d-1)}_{r_{d-2},:,:,1} & \cdots & \mathbf{A}^{(d-1)}_{r_{d-2},:,:,r_{d-1}}
    \end{bmatrix}
    \otimes
    \begin{bmatrix}
        \mathbf{A}^{(d)}_{1,:,:,1}       \\
        \vdots                           \\
        \mathbf{A}^{(d)}_{r_{d-1},:,:,1}
    \end{bmatrix}.
    \end{split}
\end{equation*}
This operation can be regarded as a generalization of the standard matrix multiplication, where the cores contain matrices as elements instead of scalar values. Just like multiplying two matrices, we compute the tensor products of the corresponding elements and then sum over the columns and rows, respectively. Below, we will use this notation to derive a compact representation of the tensor operators pertaining to nearest neighbor interaction systems.

\subsection{Nearest-Neighbor Interaction Systems}
\label{subsec: NNIS}

Following the terminology used for coupled cell systems, see e.g.~\cite{LIN2008}, we describe a \emph{nearest-neighbor interaction system} (NNIS) as a network of interacting systems -- so-called cells. These cells can represent highly diverse physical or biological systems, e.g.~coupled laser arrays~\cite{WINFUL1988}, $n$-body dynamics~\cite{GRIFFITHS1985}, chemical reaction networks~\cite{GILLESPIE03}, or heterogeneous catalytic processes \cite{GELSS2016}. Given a finite number of cells $\Theta_1, \dots, \Theta_d$ coupled in a chain or a ring, we only allow interactions/reactions $\textrm{R}$ involving one cell $\Theta_i$, $i \in \{1, \dots, d\}$, as well as reactions involving two adjacent cells $\Theta_i$ and $\Theta_{i+1}$, $i \in \{1, \dots, d-1\}$, and, if the system is cyclic, $ \Theta_d $ and $ \Theta_1 $. This coupling structure is shown in Figure~\ref{fig: nearest-neighbor reaction system}.

\begin{figure}[htb]
    \centering
    \includegraphics[width=0.4\textwidth]{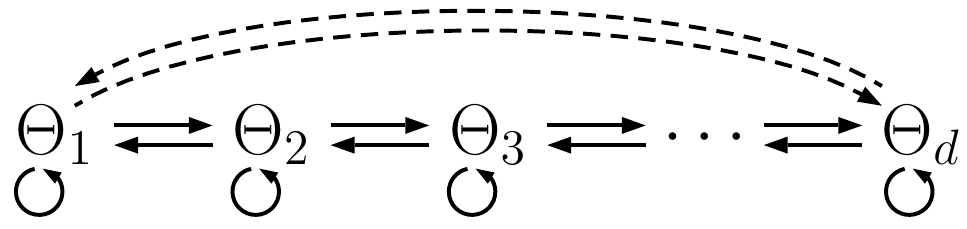}
    \caption{Visualization of nearest-neighbor interaction networks. Elementary interactions/reactions involve only one cell or two cells, respectively. Interactions between the last and the first cell are optional, i.e.~we consider cyclic and non-cyclic systems.}
    \label{fig: nearest-neighbor reaction system}
\end{figure}

We assume that the state space is finite, i.e.~each cell $\Theta_i$ can have $n_i$ different states, which are identified by the set of natural numbers $\{1, \dots, n_i\} $. Thus, the state space $ \mathcal{S} $ is given by
\begin{equation*}
    \mathcal{S} = \{1, \dots, n_1\} \times \{1 , \dots, n_2\} \times \dots \times \{1, \dots, n_d\}
\end{equation*}
and a state of the system by a vector $ X = (x_1, \dots, x_d)^T \in \mathcal{S} $, with $x_i \in \{1, \dots, n_i\}$ for $i=1, \dots , d$. A tensor $\mathbf{T} \in \R^{n_1 \times \dots \times n_d}$ based on nearest-neighbor interactions can be expressed elementwise as
\begin{equation} \label{eq: NNIS elementwise 1}
  \mathbf{T}_{x_1, \dots, x_d} = \sum_{i=1}^d \left( \mathbf{S}_i \right)_{x_i}  + \sum_{i=1}^{d-1} \left( \mathbf{K}_{i,i+1} \right)_{x_i, x_{i+1}} + \left( \mathbf{K}_{d,1} \right)_{x_d,x_1},
\end{equation}
with vectors $\mathbf{S}_i \in \R^{n_i}$ representing interactions on a single cell and matrices $\mathbf{K}_{i,i+1} \in \R^{n_i \times n_{i+1}}$ representing interactions between the cells $\Theta_i$ and $\Theta_{i+1}$. The last term in \eqref{eq: NNIS elementwise 1} is only required for cyclic NNISs, i.e.\ the matrix $\mathbf{K}_{d,1}$ is only nonzero if there is at least one interaction between $\Theta_d$ and $\Theta_1$. We will call an NNIS homogeneous if the cell types and the interactions do not depend on the cell number, i.e.\ $\mathbf{S}_1 = \mathbf{S}_2 = \dots = \mathbf{S}_d$ and $\mathbf{K}_{1,2} = \mathbf{K}_{2,3} = \dots = \mathbf{K}_{d-1,d} (= \mathbf{K}_{d,1})$, otherwise the system is called heterogeneous. Consequently, it also holds that $n_1 = n_2 = \dots = n_d$ if the system is homogeneous.

Analogously, a linear operator $\mathbf{A} \in \R^{(n_1 \times n_1) \times \dots \times (n_d \times n_d)}$ corresponding to an NNIS can be expressed elementwise as
\begin{equation}\label{eq: NNIS elementwise 2}
  \mathbf{A}_{x_1 , y_1 , \dots , x_d , y_d} = \sum_{i=1}^d \left( \mathbf{S}_i \right)_{x_i, y_i}  + \sum_{i=1}^d \left( \mathbf{K}_{i,i+1} \right)_{x_i,y_i, x_{i+1}, y_{i+1}} + \left( \mathbf{K}_{d,1} \right)_{x_d, y_d,x_1, y_1}.
\end{equation}
In this case, the components $\mathbf{S}_i$ are matrices and $\mathbf{K}_{i, i+1}$ are tensors of order 4. As already mentioned in \cite{DOLGOV01}, simple examples for tensors of this form are Ising models and linearly coupled oscillators, see Section~\ref{sec: SLIMTT}. Also more complex operators describing Markovian master equations can have the form \eqref{eq: NNIS elementwise 2}. Examples for these generators can be found in Section~\ref{sec: SLIM Markov}.

\begin{remark}\label{rem: alternative decomposition}
  Alternatively, the representation \eqref{eq: NNIS elementwise 1} may be given by
  \begin{equation}
    \mathbf{T}_{x_1, \dots, x_d} = \sum_{i=1}^d \left( \mathbf{S}_i \right)_{x_i}  + \sum_{i=1}^{d-1} \sum_{\mu = 1}^{\beta_i} \left( L_{i,\mu} \otimes M_{i+1, \mu} \right)_{x_i, x_{i+1}} + \left( L_{d,\mu} \otimes M_{1, \mu} \right)_{x_d,x_1},
  \end{equation}
  In order to obtain this representation, we only have to apply a QR-factorization or a singular value decomposition to the matrices $\mathbf{K}_{i,i+1}$ which would yield
  \begin{equation*}
    \mathbf{K}_{i,i+1} = \sum_{\mu=1}^{\beta_i} L_{i,\mu} \otimes M_{i+1,\mu},
  \end{equation*}
  with vectors $L_{i,\mu} \in \R^{n_i}$, $M_{i+1,\mu} \in \R^{n_{i+1}}$, and $\beta_i$ being the rank of the matrix $\mathbf{K}_{i,i+1}$. The same argumentation can be used for linear operators \eqref{eq: NNIS elementwise 2} in the TT format, as we will show below.
\end{remark}

\section{General SLIM Decomposition}
\label{sec: SLIMTT}

In this paper, we are particularly interested in master equations corresponding to NNISs, i.e.~computing the probability distribution of all states of a system over time. However, the TT decomposition derived here may be applied in a more general way. 

\subsection{Derivation}

Let us consider tensor trains in general. An NNIS can be represented by a tensor that has a canonical representation only consisting of elementary tensors, where at most two (adjacent) components are unequal to a vector of ones or to the identity matrix, respectively. That is, we assume that the canonical decomposition of a tensor $\mathbf{T} \in \R^{n_1 \times \dots \times n_d}$ is given by
\begin{equation} \label{eq: SLIMTT}
    \begin{split}
        \mathbf{T} &= \mathbf{S}_{1} \otimes \mathds{1}_2 \otimes \dots \otimes \mathds{1}_d \quad + \quad \dots \quad + \quad \mathds{1}_1 \otimes \dots \otimes \mathds{1}_{d-1} \otimes \mathbf{S}_{d} \\
        & \quad + \quad \sum_{\mu=1}^{\beta_1}  L_{1,\mu} \otimes M_{2,\mu} \otimes \mathds{1}_3 \otimes \dots \otimes \mathds{1}_d \quad + \quad \dots \quad + \quad \sum_{\mu=1}^{\beta_{d-1}}  \mathds{1}_1 \otimes \dots \otimes \mathds{1}_{d-2} \otimes L_{d-1,\mu} \otimes M_{d,\mu}\\
        & \quad + \quad \sum_{\mu=1}^{\beta_d}  M_{1,\mu} \otimes \mathds{1}_2 \otimes \dots \otimes \mathds{1}_{d-1} \otimes L_{d,\mu}, 
    \end{split}
\end{equation}
with $\mathds{1}_i = (1, \dots , 1)^T \in \R^{n_i}$. The last line of \eqref{eq: SLIMTT} is only required if the system is cyclic. The components $\mathbf{S}_{i}$, $L_{i,\mu}$, and $M_{i,\mu}$ are vectors in $\R^{n_i}$. If we consider a linear operator $\mathbf{A}^{(n_1 \times n_1) \times \dots \times (n_d \times n_d)}$, its canonical decomposition is given by
\begin{equation}\label{eq: SLIMTT op}
    \begin{split}
        \mathbf{A} &= \mathbf{S}_{1} \otimes I_2 \otimes \dots \otimes I_d \quad + \quad \dots \quad + \quad I_1 \otimes \dots \otimes I_{d-1} \otimes \mathbf{S}_{d} \\
        & \quad + \quad \sum_{\mu=1}^{\beta_1}  L_{1,\mu} \otimes M_{2,\mu} \otimes I_3 \otimes \dots \otimes I_d \quad + \quad \dots \quad + \quad \sum_{\mu=1}^{\beta_{d-1}}  I_1 \otimes \dots \otimes I_{d-2} \otimes L_{d-1,\mu} \otimes M_{d,\mu}\\
        & \quad + \quad \sum_{\mu=1}^{\beta_d}  M_{1,\mu} \otimes I_2 \otimes \dots \otimes I_{d-1} \otimes L_{d,\mu}, 
    \end{split}
\end{equation}
with identity matrices $I_i \in \R^{n_i \times n_i}$. In this case, the components $\mathbf{S}_{i}$, $L_{i,\mu}$, and $M_{i,\mu}$ are matrices in $\R^{n_i \times n_i}$. Again, the last line is only required for cyclic systems. Since the derivations of the TT decomposition for \eqref{eq: SLIMTT} and \eqref{eq: SLIMTT op} are almost identical, we will describe only the operator representation from now on. We gather all matrices $L_{i,\mu}$ and $M_{i+1,\mu}$ in the TT cores $\mathbf{L}_i$ and $\mathbf{M}_{i+1}$, respectively. The cores are then defined as
\begin{equation*}
    \begin{alignedat}{2}
        [ \mathbf{L}_i ] 
            & =
            \begin{bmatrix} L_{i,1} &  \dots & L_{i,\beta_i} \end{bmatrix}
            && \in \R^{1 \times n_i \times n_i \times \beta_i }, \\
        [ \mathbf{M}_{i+1} ]
            & =
            \begin{bmatrix} M_{i+1,1}  & \dots & M_{i+1,\beta_i}  \end{bmatrix}^\mathbb{T}
            && \in \R^{\beta_i \times n_{i+1} \times n_{i+1} \times 1},
    \end{alignedat}
\end{equation*}
for $ i = 1, \dots, d-1 $, and
\begin{equation*}
    \begin{alignedat}{2}
        [ \mathbf{L}_d ]
            & =
            \begin{bmatrix} L_{d,1}  & \dots & L_{d,\beta_d}  \end{bmatrix}^\mathbb{T}
            && \in \R^{ \beta_d \times n_d \times n_d \times 1}, \\
        [ \mathbf{M}_1 ]
            & =
            \begin{bmatrix} M_{1,1}  & \dots & M_{1,\beta_d}  \end{bmatrix}
            && \in \R^{1 \times n_1 \times n_1 \times \beta_d },
    \end{alignedat}
\end{equation*}
where we utilize the core-notation from \eqref{eq: core notation - single core - operator} and the definition of rank-transposed TT cores given in Appendix~\ref{app: Tensor Notation and Properties}. With the aid of the TT format, it is often possible to derive more compact representations of tensors than in the canonical format. When different rank-one tensors of a canonical representation share a number of identical cores, these elementary tensors may be expressed as one compact tensor train. Thus, the whole operator can be written in a shorter form as
\begin{equation} \label{eq: SLIMTT op 2}
    \begin{split}
        \mathbf{A} &= \mathbf{S}_1 \otimes I_2 \otimes \dots \otimes I_d  + \dots + I_1 \otimes \dots \otimes I_{d-1} \otimes \mathbf{S}_d \\
        & \quad + \left[ \mathbf{L}_1 \right] \otimes \left[ \mathbf{M}_2\right] \otimes I_3 \otimes \dots \otimes I_d + \dots + I_1 \otimes \dots \otimes I_{d-2} \otimes \left[ \mathbf{L}_{d-1} \right] \otimes \left[ \mathbf{M}_d\right]\\
        & \quad + \left[ \mathbf{M}_1\right] \otimes \left[ \mathbf{J}_2\right] \otimes \dots \otimes \left[ \mathbf{J}_{d-1} \right] \otimes \left[ \mathbf{L}_{d} \right],
    \end{split}
\end{equation}
where $\mathbf{J}_i \in \R^{\beta_d \times n_i \times n_i \times \beta_d} $ is a TT core with
\begin{equation*}
 [ \mathbf{J}_i ] = \begin{bmatrix}
        I_i  &  & 0 \\
           & \ddots & \\
           0 & & I_i
    \end{bmatrix}.
\end{equation*}
Note that $ \mathbf{J}_i $ is not a block matrix but the compact representation of a tensor. As before, the last line of \eqref{eq: SLIMTT op 2} is only required for cyclic systems. Finally, we gather the TT cores $\mathbf{S}_i$, $\mathbf{L}_i$, and $\mathbf{M}_i$ in corresponding supercores and express the linear operator $\mathbf{A}$ in the TT format in a condensed form, namely as a TT decomposition given by
\begin{equation} \label{eq: SLIMTT cyclic, heterogeneous}
\begin{split}
    \mathbf{A} & =  
    \begin{bmatrix}
        \mathbf{S}_1 & \mathbf{L}_1 & \mathbf{I}_1 & \mathbf{M}_1 
    \end{bmatrix}
    \otimes 
    \begin{bmatrix}
        \mathbf{I}_2 & 0            & 0            & 0            \\
        \mathbf{M}_2 & 0            & 0            & 0            \\       
        \mathbf{S}_2 & \mathbf{L}_2 & \mathbf{I}_2 & 0            \\
        0            & 0            & 0            & \mathbf{J}_2
    \end{bmatrix}
    \otimes \dots \otimes
    \begin{bmatrix}
        \mathbf{I}_{d-1} & 0                & 0                & 0                \\
        \mathbf{M}_{d-1} & 0                & 0                & 0                \\       
        \mathbf{S}_{d-1} & \mathbf{L}_{d-1} & \mathbf{I}_{d-1} & 0                \\
        0                & 0                & 0                & \mathbf{J}_{d-1}
    \end{bmatrix}
    \otimes
    \begin{bmatrix}
        \mathbf{I}_{d} \\
        \mathbf{M}_{d} \\       
        \mathbf{S}_{d} \\
        \mathbf{L}_{d}
    \end{bmatrix},
    \end{split}
\end{equation}
where $ \mathbf{I}_i = I_i \in \R^{n_i \times n_i}$. The above equation holds for all heterogeneous, cyclic systems. A proof can be found in Appendix~\ref{app: Properties of the SLIM Decomposition}. From now on, we will call the TT decomposition given in \eqref{eq: SLIMTT cyclic, heterogeneous} \emph{SLIM decomposition}. The origin of this term is explained by the structure of the first core. The TT ranks of the decomposition are given by 
\begin{equation*}
    r_i = 2 + \beta_i + \beta_d,
\end{equation*}
for $ i = 1, \dots, d-1 $. Furthermore, $ r_0 = r_d = 1 $. For the homogeneous case, where all $\beta_i$ are equal, the TT ranks are therefore bounded which results in a linear growth of the storage consumption with increasing dimensionality, cf.\ Appendix~\ref{app: Properties of the SLIM Decomposition}. Even considering heterogeneous systems, we can obtain a linear scaling if the ranks $\beta_i$ and dimensions $n_i$ are bounded. To reduce the storage consumption, the different TT cores can be stored as sparse arrays. An estimate of the storage consumption is given in Appendix~\ref{app: Properties of the SLIM Decomposition}. For homogeneous systems, the decomposition can be simplified to
\begin{equation} \label{eq: SLIMTT cyclic, homogeneous}
\begin{split}
    \mathbf{A} & =  
    \begin{bmatrix}
        \mathbf{S} & \mathbf{L} & \mathbf{I} & \mathbf{M} 
    \end{bmatrix}
    \otimes 
    \begin{bmatrix}
        \mathbf{I} & 0          & 0          & 0            \\
        \mathbf{M} & 0          & 0          & 0            \\       
        \mathbf{S} & \mathbf{L} & \mathbf{I} & 0            \\
        0          & 0          & 0          & \mathbf{J}
    \end{bmatrix}
    \otimes \dots \otimes
    \begin{bmatrix}
        \mathbf{I} & 0          & 0          & 0            \\
        \mathbf{M} & 0          & 0          & 0            \\       
        \mathbf{S} & \mathbf{L} & \mathbf{I} & 0            \\
        0          & 0          & 0          & \mathbf{J}
    \end{bmatrix}
    \otimes
    \begin{bmatrix}
        \mathbf{I} \\
        \mathbf{M} \\       
    \mathbf{S} \\
    \mathbf{L}
    \end{bmatrix}.
    \end{split}
\end{equation}
A beneficial property of the SLIM decomposition of homogeneous systems is the arbitrary scaling of the interaction system, i.e.~decreasing or increasing the number of cells corresponds to removing or inserting a TT core. Only the first and the last core in \eqref{eq: SLIMTT cyclic, homogeneous} have to be fixed. The number of cores in between can differ, but must be greater than 0 (since cyclic systems have at least 3 cells, two linked cells are represented by \eqref{eq: SLIMTT non-cyclic, heterogeneous}).

If we consider a non-cyclic system, either homogeneous or heterogeneous, interactions between the last and the first cell are omitted. In particular, the SLIM decomposition for non-cyclic, heterogeneous systems is given by
\begin{equation} \label{eq: SLIMTT non-cyclic, heterogeneous}
\begin{split}
    \mathbf{A} & =  
    \begin{bmatrix}
        \mathbf{S}_1 & \mathbf{L}_1 & \mathbf{I}_1
    \end{bmatrix}
    \otimes 
    \begin{bmatrix}
        \mathbf{I}_2 & 0            & 0                       \\
        \mathbf{M}_2 & 0            & 0                       \\       
                \mathbf{S}_2 & \mathbf{L}_2 & \mathbf{I}_2 
    \end{bmatrix}
    \otimes \dots \otimes
    \begin{bmatrix}
        \mathbf{I}_{d-1} & 0                & 0                 \\
        \mathbf{M}_{d-1} & 0                & 0                 \\       
        \mathbf{S}_{d-1} & \mathbf{L}_{d-1} & \mathbf{I}_{d-1}
    \end{bmatrix}
    \otimes
    \begin{bmatrix}
        \mathbf{I}_{d} \\
        \mathbf{M}_{d} \\       
        \mathbf{S}_{d} 
    \end{bmatrix}.
\end{split}
\end{equation}

\subsection{Examples}

\subsubsection{Ising Model}

The first example is taken from the field of statistical mechanics. Ising models were proposed by Wilhelm Lenz \cite{LENZ1920} and first studied in detail by Ernst Ising \cite{ISING1925}. Consisting of $ d $ cells that represent magnetic dipole moments of atomic spins, Ising models describe ferromagnetic effects in solids. The spin corresponding to each cell can be in two states, $+1$ or $-1$. Usually the cells are arranged in an $d$-dimensional lattice where only adjacent cells interact. Ising models are of particular interest since they can be solved exactly.

Interactions are either between two adjacent cells or on one cell. Let $ x_i = \pm 1 $ be the state of cell $\Theta_i$. The spin configuration of the whole system is then given by $X = (x_1, \dots, x_d)^T$. Furthermore, we denote by $\mathcal{N}$ the set of all pairs of indices corresponding to nearest neighbors, i.e.~$(i,j) \in \mathcal{N}$ if $\Theta_i$ and $\Theta_j$ are adjacent cells. We consider the Hamiltonian function $ H $ with
\begin{equation*}
    H(X) = - \sum_{(i,j) \in \mathcal{N}} J_{ij} x_i x_j - \mu \sum_{i = 1}^{d} h_i x_i .
\end{equation*}
The equation above represents the energy of the spin configuration $X$. The interaction strength between two adjacent cells $\Theta_i$ and $\Theta_j$ is denoted by $J_{ij}$, $\mu$ and $h_i$ represent the magnetic moment and an external magnetic field, respectively.

\begin{figure}[htb]
    \centering
    \includegraphics[width=0.4\textwidth]{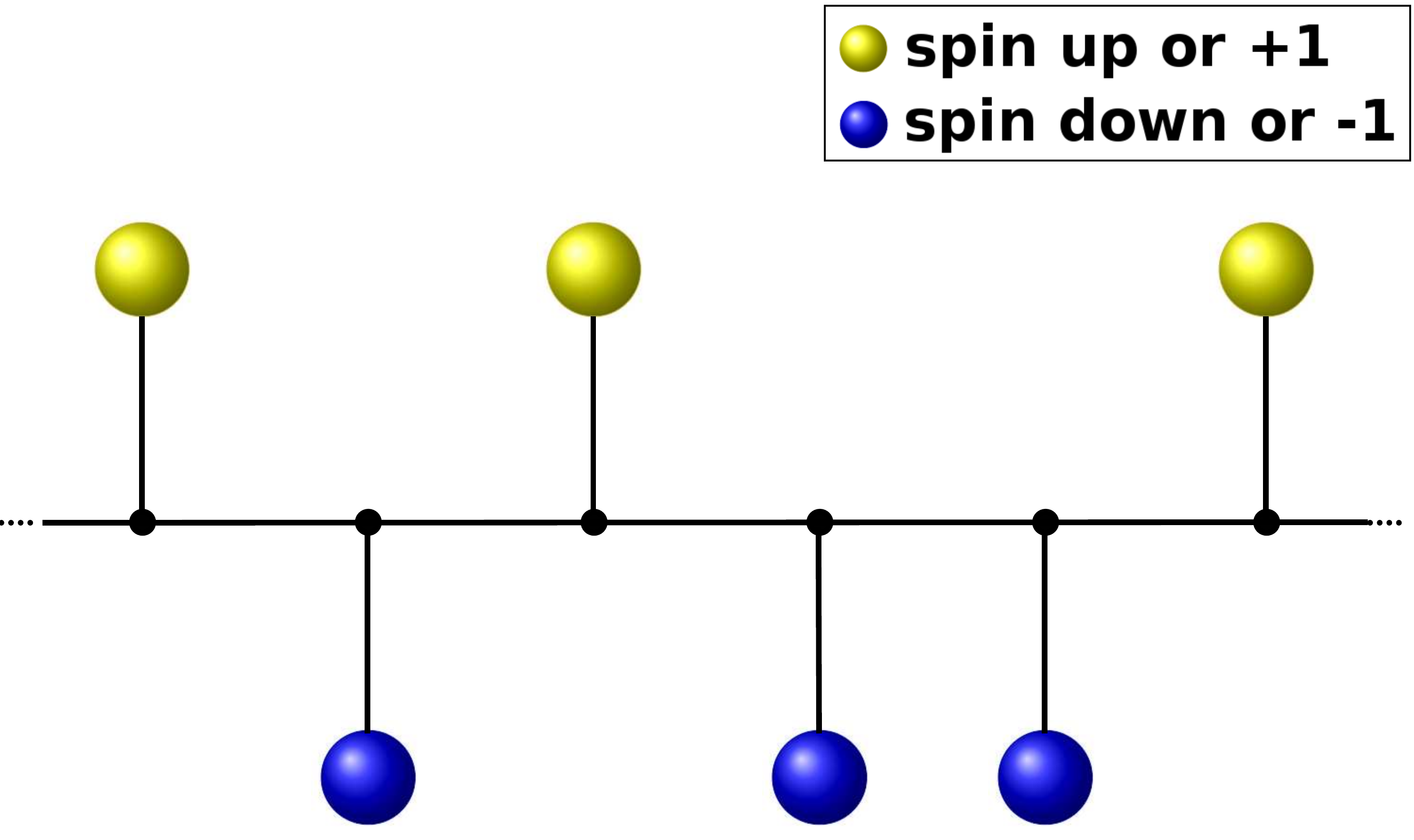}
    \caption{Pictorial representation of the Ising model.}
    \label{fig: Ising model}
\end{figure}

Here, we focus on the one-dimensional Ising model (see Figure~\ref{fig: Ising model}), which is represented by a cyclic, homogeneous NNIS, i.e.\ we consider $d$ cells $ \Theta_1, \dots, \Theta_d$ arranged in a ring. It is common to simplify the Ising model by assuming the same interaction strength between all nearest neighbors, i.e.\ $J_{ij} = J$ for all $(i,j) \in \mathcal{N}$, and a constant external magnetic field, i.e.\ $h_i = h$ for $i = 1, \dots, d$. The Hamiltonian of the one-dimensional Ising model then becomes 
\begin{equation*}
    H(X) = - \sum_{i=1}^{d} J x_i x_{i+1} - \mu \sum_{i = 1}^{d} h x_i ,
\end{equation*}
with $x_{d+1} = x_1$. Now, we consider the Hamiltonian function as a tensor $\mathbf{H} \in \R^{2 \times \dots \times 2}$. That is, all indices can take the value $1$ and $2$, respectively, with
\begin{equation*}
    \mathbf{H}_{y_1, \dots, y_d} = H(x_1, \dots, x_d),
\end{equation*}
where
\begin{equation*}
    x_i=
    \begin{cases}
        +1,~\text{if}~y_i=1,\\
        -1,~\text{if}~y_i=2.
    \end{cases}
\end{equation*}
Using the canonical format, we can express the tensor $\mathbf{H}$ as
{ \allowdisplaybreaks
\begin{align*}
    \mathbf{H} =& - \mu h \cdot
    \begin{pmatrix}
        + 1 \\ -1
    \end{pmatrix}
    \otimes
    \mathds{1}
    \otimes
    \dots
    \otimes
    \mathds{1}  -  \mu h \cdot
    \mathds{1}
    \otimes
    \dots
    \otimes
    \mathds{1}
    \otimes
    \begin{pmatrix}
        + 1 \\ -1
    \end{pmatrix} \\
    & -J \cdot 
    \begin{pmatrix}
        + 1 \\ -1
    \end{pmatrix} 
    \otimes 
    \begin{pmatrix}
        + 1 \\ -1
    \end{pmatrix} 
    \otimes
    \mathds{1}
    \otimes
    \dots
    \otimes 
    \mathds{1} - \dots \\
		& \dots - J \cdot 
    \mathds{1}
    \otimes 
    \dots 
    \otimes
    \mathds{1}
    \otimes
    \begin{pmatrix}
        + 1 \\ -1
    \end{pmatrix} 
    \otimes
    \begin{pmatrix}
        + 1 \\ -1
    \end{pmatrix} \\
    &-J \cdot 
    \begin{pmatrix}
        + 1 \\ -1
    \end{pmatrix}
    \otimes
    \mathds{1}
    \otimes
    \dots 
    \otimes
    \mathds{1}
    \otimes
    \begin{pmatrix}
        + 1 \\ -1
    \end{pmatrix}.
\end{align*}}

By defining $\mathbf{S} = - \mu h \cdot (+ 1 , -1 )^T$, $\mathbf{L} = - J \cdot ( + 1 , -1 )^T$, $\mathbf{I} = (1,1)^T$, and $\mathbf{M} = ( + 1 , -1 )^T$, we can express $\mathbf{H}$ as a SLIM decomposition similar to \eqref{eq: SLIMTT cyclic, homogeneous}:
\begin{equation*}
\begin{split}
    \mathbf{H} & =  
    \begin{bmatrix}
        \mathbf{S} & \mathbf{L} & \mathbf{I} & \mathbf{M} 
    \end{bmatrix}
    \otimes 
    \begin{bmatrix}
        \mathbf{I} & 0          & 0          & 0            \\
        \mathbf{M} & 0          & 0          & 0            \\       
        \mathbf{S} & \mathbf{L} & \mathbf{I} & 0            \\
        0          & 0          & 0          & \mathbf{J}
    \end{bmatrix}
    \otimes \dots \otimes
    \begin{bmatrix}
        \mathbf{I} & 0          & 0          & 0            \\
        \mathbf{M} & 0          & 0          & 0            \\       
        \mathbf{S} & \mathbf{L} & \mathbf{I} & 0            \\
        0          & 0          & 0          & \mathbf{J}
    \end{bmatrix}
    \otimes
    \begin{bmatrix}
        \mathbf{I} \\
        \mathbf{M} \\       
    \mathbf{S} \\
    \mathbf{L}
    \end{bmatrix}.
    \end{split}
\end{equation*}

\subsubsection{Linearly Coupled Oscillator}

The second example for a general application of the SLIM decomposition is a quantum system consisting of a one-dimensional chain of $d$ identical harmonic oscillators coupled by springs, see e.g.~\cite{LIEVENS2006, PLENIO2004}. Figure~\ref{fig: oscillators} shows an illustration of this quantum system. Assuming periodic boundary conditions, the Hamiltonian operator of the system is given by
\begin{equation} \label{eq: hamiltonian oscillator}
  \hat{H} = \sum_{i=1}^{d} \left( \frac{1}{2m}  \hat{p}^2_i + \frac{m \omega^2}{2}  \hat{x}_i^2 + \frac{c m}{2}(\hat{x}_i - \hat{x}_{i+1})^2 \right),
\end{equation}
where $m$ is the mass of each oscillator, $\omega$ its natural frequency, and $c$ the coupling strength. 

\begin{figure}[htb]
    \centering
    \includegraphics[width=0.4\textwidth]{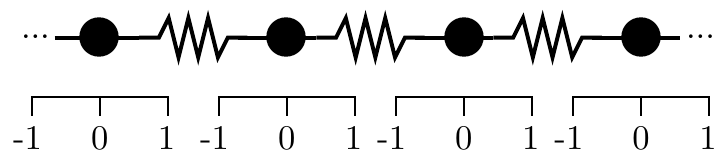}
    \caption{Visualization of linearly coupled oscillators, where all masses are in equilibrium position. Displacements $x_i$ are between $-1$ and $1$.}
    \label{fig: oscillators}
\end{figure}

\noindent The displacement of the $i$-th mass point with respect to its equilibrium position is measured by the position operator $\hat{x}_i$, while $\hat{p}_i$ represents the momentum operator of the $i$-th mass point. Applied to a wave function $\psi_i$ at position $x_i$ of the $i$-th oscillator, these operators can be expressed as
\begin{equation*}
  \hat{p}_i \psi_i(x_i) = -\imath \hslash \frac{\partial}{\partial x} \psi_i(x_i) \qquad \textrm{and} \qquad \hat{x}_i \psi_i(x_i) = x_i \cdot \psi_i(x_i),
\end{equation*}
where $\imath$ denotes the imaginary unit. Using the finite difference method with mesh width $h = 1/m$, $m \in \mathbb{N}$, we obtain
\begin{equation}
  \hat{p}_i^2 \psi_i(x_i) \approx -\hslash^2 \frac{\psi_i(x_i-h) - 2 \psi_i(x_i) + \psi_i(x_i+h)}{h^2}.
\end{equation}
Defining the discrete displacements $x_{i,k} = k \cdot h$ for $k = - m , \dots ,0, \dots , m$ and $\psi_{i,k}$ representing the numerical approximation of $\psi_i(x_{i,k})$, this yields
\begin{equation}\label{eq: discrete momentum}
  \left( \frac{\hslash}{h} \right)^2 
  \underbrace{\begin{pmatrix}
    \phantom{-}2 & -1           &        & \\
    -1           & \phantom{-}2 & \ddots     & \\
                 & \ddots       & \ddots & - 1          \\
                 &              & -1     & \phantom{-}2
  \end{pmatrix}}_{=: D_{p}}
  \begin{pmatrix}
    \psi_{i,-m} \\
    \vdots \\
    \psi_{i,m}
  \end{pmatrix}.
\end{equation}
Discretizing $\hat{x}_i$, we obtain
\begin{equation}\label{eq: discrete position}
  \underbrace{\begin{pmatrix}
    x_{i,-m} &  & \\
           & \ddots & \\
           &        & x_{i,m}
  \end{pmatrix}}_{=:D_x}
  \begin{pmatrix}
    \psi_{i,-m} \\
    \vdots \\
    \psi_{i,m}
  \end{pmatrix}.
\end{equation}
Since the momentum and position operators $\hat{p}_i$ and $\hat{x}_i$ act only on the wave function corresponding to the cell/oscillator $\Theta_i$, we can express these operators as rank-one tensors containing the matrices from \eqref{eq: discrete momentum} and \eqref{eq: discrete position}, respectively, as a component. Representing the discretization of the position and momentum operators as tensor decompositions $\hat{\mathbf{p}}_i$ and $\hat{\mathbf{x}}_i$, we can write
\begin{equation*}
  \begin{split}
    \hat{\mathbf{p}}^2_i &= \left( \frac{\hslash}{h} \right)^2  \cdot I \otimes \dots \otimes I \otimes
    \underbrace{D_p}_{\shortstack{\scriptsize\textit{i}{-th}\\\scriptsize{component}}}
    \otimes I \otimes \dots \otimes I, \\
    \hat{\mathbf{x}}^2_i & = I \otimes \dots \otimes I \otimes
    \underbrace{D_x^2}_{\shortstack{\scriptsize\textit{i}{-th}\\\scriptsize{component}}}
    \otimes I \otimes \dots \otimes I.
        \end{split}
\end{equation*}
The rules for tensor multiplication then imply that we can write $\left( \hat{\mathbf{x}}_i - \hat{\mathbf{x}}_{i+1} \right)^2$, compare \eqref{eq: hamiltonian oscillator}, as
\begin{equation*}
  \begin{split}        
    \left( \hat{\mathbf{x}}_i - \hat{\mathbf{x}}_{i+1} \right)^2 & = \hat{\mathbf{x}}_i^2 \quad + \quad  \hat{\mathbf{x}}_{i+1}^2\\
		& \quad - \quad 2 \cdot I \otimes \dots \otimes I \otimes
    \underbrace{D_x}_{\shortstack{\scriptsize\phantom{(}\textit{i}{-th}\phantom{)}\\\scriptsize{component}}} \otimes
     \underbrace{D_x}_{\shortstack{\scriptsize{(}\textit{i}{+1)-th}\\\scriptsize{component}}}
    \otimes I \otimes \dots \otimes I .
  \end{split}
\end{equation*}
Finally, by defining 
\begin{equation*}
    \mathbf{S} = \frac{\hslash^2}{2 m h^2}
  D_p
  + \frac{m \omega^2}{2} 
  D_x^2
    + cm
  D_x^2, \quad \mathbf{L}  = -cm 
    D_x, 
    \quad \mathbf{M}  =  
    D_x,
\end{equation*}
we can give the SLIM decomposition of the discretized Hamiltonian $\hat{\mathbf{H}}$ as
\begin{equation*}
\begin{split}
    \hat{\mathbf{H}} & =  
    \begin{bmatrix}
        \mathbf{S} & \mathbf{L} & \mathbf{I} & \mathbf{M} 
    \end{bmatrix}
    \otimes 
    \begin{bmatrix}
        \mathbf{I} & 0          & 0          & 0            \\
        \mathbf{M} & 0          & 0          & 0            \\       
        \mathbf{S} & \mathbf{L} & \mathbf{I} & 0            \\
        0          & 0          & 0          & \mathbf{J}
    \end{bmatrix}
    \otimes \dots \otimes
    \begin{bmatrix}
        \mathbf{I} & 0          & 0          & 0            \\
        \mathbf{M} & 0          & 0          & 0            \\       
        \mathbf{S} & \mathbf{L} & \mathbf{I} & 0            \\
        0          & 0          & 0          & \mathbf{J}
    \end{bmatrix}
    \otimes
    \begin{bmatrix}
        \mathbf{I} \\
        \mathbf{M} \\       
    \mathbf{S} \\
    \mathbf{L}
    \end{bmatrix}.
    \end{split}
\end{equation*}

\section{SLIM Decomposition for Markov Generators}
\label{sec: SLIM Markov}

In this section, we will consider tensor representations of Markovian master equations and illustrate the results using a simple guiding example. Consider a continuous-time Markov jump process on the state space $\mathcal{S}$. Let $P(X,t)$ be the probability that the system is in state $X$ at time $t$ under the condition that it was in state $X_0$ at time $t_0$. For the sake of simplicity, the dependence on the initial state is omitted. The probability distribution $P(X,t)$ then obeys an MME~\cite{KAMPEN}, given by
\begin{equation*}
    \pd{}{t} P(X,t) = \sum_{Y} W(X | Y) P(Y,t) - \sum_Y W(Y | X) P(X,t),
\end{equation*}
where $W(Y | X)$ is the transition rate to go from state $X$ to state $Y$ by an elementary reaction as described above. Note that $W(Y | X)$ is only nonzero if $X$ and $Y$ are elements of the state space $\mathcal{S}$ and there is an elementary reaction $\textrm{R}_\mu$ involving both states. If we denote the net changes in the state vector $X$ caused by a single firing of $\textrm{R}_\mu$ by the vector $\xi_\mu \in \mathbb{Z}^d$, the \emph{reaction propensity} $a_\mu$ is given by
\begin{equation*}
    a_\mu(X) = W(X+\xi_\mu \vert X),
\end{equation*}
which is only nonzero if $X$ satisfies the requirements that $\textrm{R}_\mu$ may fire. Thus, summing over all $ \mathcal{M} $ allowed reactions $ \textrm{R}_1, \dots, \textrm{R}_\mathcal{M} $, we obtain
\begin{equation} \label{MME as CME}
    \pd{}{t} P(X,t) = \sum_{\mu=1}^\mathcal{M} a_\mu(X-\xi_\mu)P(X-\xi_\mu,t) - a_\mu(X)P(X,t).
\end{equation}
We assume that all reaction events are local, i.e.~an event will only change the configuration in the vicinity of a particular cell. Thus, the number of elementary reactions we have to consider is bounded and the number $\mathcal{M}$ will be very small compared to the size of the state space.

Equation \eqref{MME as CME} has the same structure as a so-called \emph{chemical master equation} (CME) \cite{GILLESPIE01}, a special type of an MME which describes the time-evolution of a chemical system. However, in our case, the state space does not represent numbers of different molecules, instead it denotes the more general configuration of the cells. Due to the summation in \eqref{MME as CME}, we consider exactly all possible states from which $ X $ can be reached and all states that can be reached from X by a single firing of one of the elementary reactions $\textrm{R}_\mu$.

From now on, we consider a Markovian master equation defined on an NNIS. The elementary reactions occurring in such systems are of the form
\begin{equation}\label{eq: reactions}
  \begin{aligned}
      & \text{(i)}   & \quad & \textrm{R}_\mu: x_i          & \rightarrow \quad & y_i, \\
      & \text{(ii)}  & \quad & \textrm{R}_\mu: x_i, x_{i+1} & \rightarrow \quad & y_i, y_{i+1}, \quad\text{and} \\
      & \text{(iii)} & \quad & \textrm{R}_\mu: x_d, x_1     & \rightarrow \quad & y_d, y_1,
  \end{aligned}
\end{equation}
respectively, with $x_i , y_i \in \{1, \dots, n_i\}$ for $i = 1, \dots, d$. That is, each reaction only changes the state of one cell or of two adjacent cells. We will call these types \emph{single-cell reactions} (SCR) and \emph{two-cell reactions} (TCR). TCRs of the form (iii) only occur in cyclic systems.

\begin{example} \label{ex: guiding}
  \begin{figure}[H]
      \centering
      \includegraphics[width=0.4\textwidth]{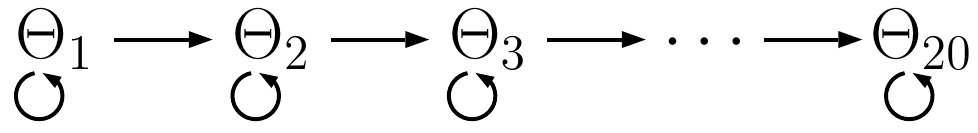}
      \caption{Visualization of the signal cascade model. Elementary reactions on one cell are creation and destruction reactions. Elementary reactions proceed only in one direction and represent the creation of one protein depending on the number of the preceding proteins.}
      \label{fig: signaling cascade}
  \end{figure}

  As a simple example of an NNIS, we consider a cascading process on a genetic network consisting of 20 genes, see~\cite{DOLGOV01, HEGLAND2007, AMMAR2011}. Here, the cells represent adjacent genes producing proteins that affect the expression of subsequent genes. The state of a cell describes the number of such proteins. In Figure \ref{fig: signaling cascade}, the structure of this system is shown. The reactions and reaction propensities are:
  
  \vspace{0.25cm}
  \noindent \textit{Creation of the first protein corresponding to} $\Theta_1$:
  \begin{equation*}
    \textrm{R}_1: x_1 \rightarrow x_1 +1, \quad  a_1(X) = 0.7.
  \end{equation*}

  \noindent \textit{Creation of a protein corresponding to} $\Theta_i$, $2 \leq i \leq d$:
  \begin{equation*}
    \textrm{R}_i: x_{i-1}, x_i  \rightarrow x_{i-1}, x_i +1, \quad  a_i(X) = \frac{x_{i-1}}{5 + x_{i-1}}.
  \end{equation*}
  
  \noindent \textit{Destruction of a protein corresponding to} $\Theta_i$, $1 \leq i \leq d$:
  \begin{equation*}
    \textrm{R}_{i+d}:  x_i  \rightarrow  x_i - 1, \quad  a_{i+d}(X) = 0.07 \, x_i.
  \end{equation*}
  
  If we start with an initial state where all numbers of proteins are 0, the value of the probability density function for any $x_i \geq 63$ is below machine precision for all times $t \geq 0$, see \cite{DOLGOV01}. Thus, we can consider a finite state space 
  \begin{equation*}
    \mathcal{S} = \{0, \dots, 63\} \times \dots \times \{0, \dots, 63\} .
  \end{equation*}
  We map this state space to $\mathcal{S}' = \{1, \dots, 64\} \times \dots \times \{1, \dots, 64\}$ since we identify the states of each cell by a set of natural numbers as mentioned above. In this way, we prevent conflicts with the later introduced tensor indexing notation. The model is non-cyclic and heterogeneous since the first creation reaction differs from the other creation reactions. An exact TT decomposition of the corresponding MME operator of this system can be found in \cite{DOLGOV01}. That decomposition, however, differs from the generally applicable SLIM decomposition as we will show in the following sections. \exampleSymbol
\end{example}

\subsection{Tensor Representation of the Markovian Master Equation}

In this section, we show how to derive tensor-based expressions of MMEs written as CMEs and introduce corresponding quantities such as, for instance, multidimensional shift operators. For an analogous derivation of the tensor representation of a CME, see \cite{DOLGOV01}. Let $\mathcal{M}$ be the number of all elementary reactions involving one or two cells. We identify each propensity function $a_\mu: \mathcal{S} \rightarrow \R$ with a tensor $\mathbf{a}_\mu \in \R^{n_1 \times \dots \times n_d}$, i.e.\ for a state $X =(x_1 , \dots , x_d)^T \in \mathcal{S}$, we define 
\begin{equation*}
  \left( \mathbf{a}_\mu \right)_{x_1 , \dots , x_d} = a_\mu(X),
\end{equation*}
for $\mu = 1 , \dots , \mathcal{M}$. These propensity tensors can then be expressed in the canonical format as
\begin{equation*}
  \mathbf{a}_\mu = \sum_{k=1}^{r_\mu} \left( \mathbf{a}_\mu^{(1)}\right)_{k,:} \otimes \dots \otimes \left( \mathbf{a}_\mu^{(d)}\right)_{k,:},
\end{equation*}
with cores $\mathbf{a}^{(i)}_\mu \in \R^{r_\mu \times n_i}$. Here, we only rely on the fact that $\mathbf{a}_\mu$ can be expressed as a canonical tensor without taking the rank $r_\mu$ into account. Furthermore, we describe the probabilities $ P(X,t) $ by a tensor $\mathbf{P}(t) \in \mathbb{R}^{n_1 \times \dots \times n_d}$, with
\begin{equation*}
    (\mathbf{P} (t))_{x_1, \dots, x_d} = P(X,t).
\end{equation*}

\begin{definition}\label{def: shift operators}
Let $G_i(k) \in \mathbb{R}^{n_i \times n_i}$ denote the shift matrix which is given by $(G_i(k))_{x,y} := \delta_{y-x, k}$ with $\delta_{y-x, k}$ representing the Kronecker delta\footnote{Note that $G_i(0)$ is then simply the identity matrix in $\mathbb{R}^{n_i \times n_i}$.}. Then the \emph{multidimensional shift operators} $\mathbf{G}_{\mu}$ and $\mathbf{G}_0$ are defined as
\begin{equation*}
    \mathbf{G}_\mu= G_1(-\xi_{\mu} (1)) \otimes \dots \otimes G_d(-\xi_{\mu}(d))
\end{equation*}
and
\begin{equation*}
    \mathbf{G}_0= G_1(0) \otimes \dots \otimes G_d(0) =: \mathbf{I}.
\end{equation*}
\end{definition}

With the aid of this definition, we can now reformulate \eqref{MME as CME} in a more compact way as
\begin{equation} \label{cme_tensor_notation}
    \pd{}{t} \textbf{P}(t) = \left ( \sum_{\mu=1}^{\mathcal{M}}  (\mathbf{G}_{\mu} - \mathbf{I} ) \cdot \diag(\textbf{a}_{\mu} )\right ) \cdot \textbf{P}(t),
\end{equation}
where we define $\diag(\textbf{a}_{\mu})$ to be the tensor product of matrices containing the entries of $\left( \textbf{a}_{\mu}^{(1)} \right)_{k,:}, \dots, \left( \textbf{a}_{\mu}^{(d)} \right)_{k,:}$ as diagonals, i.e.
\begin{equation*}
    \diag(\textbf{a}_{\mu}) = \sum_{k=1}^{r_\mu} \diag\left( \left( \textbf{a}_{\mu}^{(1)} \right)_{k,:} \right) \otimes \dots \otimes \diag \left( \left( \textbf{a}_{\mu}^{(d)} \right)_{k,:} \right),
\end{equation*}
for $ \mu = 1, \dots, \mathcal{M} $. A proof that this notation is equivalent to the master equation given in \eqref{MME as CME} can be found in Appendix~\ref{app: Tensor Notation and Properties} as well as in~\cite{GELSS2016}. In what follows, let
\begin{equation} \label{equation:TTME:operator}
    \mathbf{A} = \sum_{\mu=1}^{\mathcal{M}} (\mathbf{G}_{\mu} - \mathbf{I} ) \cdot \diag(\textbf{a}_{\mu}),
\end{equation}
so that \eqref{cme_tensor_notation} can be written as $ \pd{}{t} \textbf{P}(t) = \mathbf{A} \cdot \textbf{P}(t)$. The aim is then to solve the tensor-based MME numerically using implicit integration schemes. The resulting systems of linear equations can be solved, for instance, with ALS~\cite{HOLTZ02}.

\begin{example}\label{ex: cascade 2}
Let us consider our guiding example defined above and illustrate what the reaction propensities, the vectors of net changes, and the shift operators look like in this case. For this purpose, we split the reactions into creation and destruction operations, i.e.~the reaction $\textrm{R}_\mu$, $\mu = 1, \dots, d$, represents the creation of the protein corresponding to cell $\Theta_\mu$ and for $\mu = d+1, \dots, 2d$, we consider the destruction of the proteins, see Example \ref{ex: guiding}. Written as rank-one tensors, the reaction propensities have the following form:
\begin{align*}
\mathbf{a}_1 &= 0.7 \cdot \mathds{1}  \otimes \dots \otimes \mathds{1}, 
    & \mathbf{a}_{d+1} & = \begin{pmatrix} 0.07 \cdot 0 \\ \vdots \\ 0.07 \cdot 63 \end{pmatrix} \otimes  \mathds{1} \otimes  \dots  \otimes \mathds{1}, \\[1ex]
    \mathbf{a}_2 &= \begin{pmatrix} \frac{0}{5+0} \\ \vdots \\ \frac{63}{5+63} \end{pmatrix}\otimes \mathds{1} \otimes \dots  \otimes \mathds{1},
    & \mathbf{a}_{d+2} & = \mathds{1} \otimes \begin{pmatrix} 0.07 \cdot 0 \\ \vdots \\ 0.07 \cdot 63 \end{pmatrix} \otimes  \mathds{1} \otimes \dots \otimes \mathds{1}, \\
    & \vdots && \vdots \\
    \mathbf{a}_d &= \mathds{1} \otimes \dots \otimes \mathds{1} \otimes  \begin{pmatrix} \frac{0}{5+0} \\ \vdots \\ \frac{63}{5+63} \end{pmatrix} \otimes \mathds{1},
    & \mathbf{a}_{2d} & = \mathds{1}  \otimes \dots \otimes \mathds{1} \otimes \begin{pmatrix} 0.07 \cdot 0 \\ \vdots \\ 0.07 \cdot 63 \end{pmatrix}.
\end{align*}
The vectors of net changes are all zero except for one entry, i.e.
\begin{align*}
    \xi_1 &= \begin{pmatrix} 1 & 0 & \cdots & 0 \end{pmatrix},
    & \xi_{d+1} &= \begin{pmatrix} -1 & 0 & \cdots & 0 \end{pmatrix}, \\
    & \vdots && \vdots \\
    \xi_d &= \begin{pmatrix} 0 & \cdots & 0 & 1\end{pmatrix},
    & \xi_{2d} &=  \begin{pmatrix} 0 & \cdots & 0 & -1\end{pmatrix}.
\end{align*}
By defining the shift matrices $G^{\downarrow} := G_i(-1)$ and $G^{\uparrow} := G_i(1)$ for $i = 1, \dots , d$, i.e.
\begin{align*}
    G^{\downarrow} = \begin{pmatrix} 0 &  & & 0\\ 1 & 0  & &  \\ & \ddots & \ddots \\ 0 & & 1 & 0 \end{pmatrix} \qquad \textrm{and} \qquad G^{\uparrow} = \begin{pmatrix} 0 & 1 & & 0\\  & 0  & \ddots &  \\ &  & \ddots & 1 \\ 0 & &  & 0 \end{pmatrix},
\end{align*}
the corresponding shift operators for the creation and destruction reactions are given by
\begin{align*}
\mathbf{G}_1 = G^{\downarrow} \otimes I \otimes \dots \otimes I, \quad \dots, \quad \mathbf{G}_d = I \otimes \dots \otimes I \otimes G^{\downarrow},
\end{align*}
and
\begin{align*}
\mathbf{G}_{d+1} = G^{\uparrow} \otimes I \otimes \dots \otimes I, \quad \dots, \quad \mathbf{G}_{2d} = I \otimes \dots \otimes I \otimes G^{\uparrow}.
\end{align*}
We will describe the resulting SLIM decomposition in the next section. \exampleSymbol
\end{example}

\subsection{Derivation}
\label{subsec: SLIM}

As mentioned above, for the nearest-neighbor interaction networks considered in this paper, an elementary reaction involves only one or two cells, respectively, i.e.~elementary reactions either depend on the state of a single cell or on the states of two adjacent cells, see \eqref{eq: reactions}. This implies that any reaction propensity that belongs to an SCR $\mathrm{R}_{i,\mu}$ on cell $\Theta_i$ has the form
\begin{equation} \label{single-cell propensity}
    \mathbf{a}_{i,\mu} = \mathds{1}_1 \otimes \dots \otimes \mathds{1}_{i-1} \otimes \mathrm{a}_{i,\mu} \otimes \mathds{1}_{i+1} \otimes \dots \otimes \mathds{1}_d,
\end{equation}
with $\mathrm{a}_{i,\mu} \in \R^{n_i}$ and $\mu=1, \dots , \alpha_i$, where $\alpha_i \in \mathbb{N}$ is the number of all SCRs on $\Theta_i$. For the two-cell propensities $\mathbf{a}_{i,i+1,\mu}$ corresponding to TCRs $\mathrm{R}_{i,i+1,\mu}$ on the cell pairs $\Theta_i$ and $\Theta_{i+1}$, $i = 1, \dots , d-1$, we can write
\begin{equation} \label{two-cell propensity}
    \mathbf{a}_{i,i+1,\mu} = \mathds{1}_1 \otimes \dots \otimes \mathds{1}_{i-1} \otimes \mathrm{a}_{i,i+1,\mu}  \otimes \mathds{1}_{i+2} \otimes \dots \otimes \mathds{1}_d,
\end{equation}
with $\mathrm{a}_{i,i+1,\mu} \in \R^{n_i \times n_{i+1}}$. We assume that there are $\beta_i$ TCRs between the cells $\Theta_i$ and $\Theta_{i+1}$, i.e.\ $\mu = 1, \dots , \beta_i$. As already mentioned in Remark \ref{rem: alternative decomposition}, we can decompose $\mathrm{a}_{i,i+1,\mu }$ into canonical tensor cores by applying, for instance, a singular value decomposition\footnote{Assume that $ \mathrm{a}_{i,i+1,\mu} = U \Sigma V^T = \sum_{k=1}^r \sigma_k u_k \otimes v_k $, then $ \left( \mathrm{a}_{i,i+1,\mu}^{(1)} \right)_{k,:} = \sigma_k u_k $ and $ \left( \mathrm{a}_{i,i+1,\mu}^{(2)} \right)_{k,:} = v_k $.} or QR-factorization, i.e. 
\begin{equation*}
  \mathrm{a}_{i,i+1,\mu} = \sum_{k=1}^{r_{i,i+1,\mu}} \left( \mathrm{a}_{i,i+1,\mu}^{(1)} \right)_{k,:} \otimes  \left( \mathrm{a}_{i,i+1,\mu}^{(2)} \right)_{k,:}.
\end{equation*}
For cyclic systems, we consider the permuted propensity tensor $\tilde{\mathbf{a}}_{d,1,\mu}$ with $\left( \tilde{\mathbf{a}} _{d,1,\mu}\right)_{x_2, \dots , x_d, x_1} = \left( \mathbf{a}_{d,1,\mu} \right)_{x_1 , x_2 ,\dots , x_d}$, see Remark~\ref{re: cyclic permutation}. This tensor can then be written as
\begin{equation*}
  \tilde{\mathbf{a}}_{d,1,\mu} = \mathds{1}_2 \otimes \dots \otimes \mathds{1}_{d-1} \otimes \mathrm{a}_{d,1,\mu},
\end{equation*}
for $\mu = 1, \dots, \beta_d$, with $\beta_d$ being the number of TCRs between the cells $\Theta_d$ and $\Theta_1$. Again, we can decompose $\mathrm{a}_{d,1,\mu}$ and write
\begin{equation} \label{eq: propensity cyclic - 1}
  \tilde{\mathbf{a}}_{d,1,\mu} = \sum_{k=1}^{r_{d,1,\mu}} \mathds{1}_2 \otimes \dots \otimes \mathds{1}_{d-1} \otimes \left( \mathrm{a}_{d,1,\mu}^{(1)} \right)_{k,:} \otimes \left( \mathrm{a}_{d,1,\mu}^{(2)} \right)_{k,:}.
\end{equation}
As stated in Remark \ref{re: cyclic permutation}, a cyclic permutation of the cores corresponds to a permutation of the indices. Thus, we have
\begin{equation} \label{eq: propensity cyclic - 2}
  \mathbf{a}_{d,1,\mu} = \sum_{k=1}^{r_{d,1,\mu}}  \left( \mathrm{a}_{d,1,\mu}^{(2)} \right)_{k,:} \otimes \mathds{1}_2 \otimes \dots \otimes \mathds{1}_{d-1} \otimes \left( \mathrm{a}_{d,1,\mu}^{(1)} \right)_{k,:} ,
\end{equation}
for $\mu = 1, \dots , \beta_d$. That is, only the components corresponding to one cell or to two adjacent cells of each reaction propensity are unequal to a vector of ones and we obtain
\begin{equation*}
    \diag\left( \mathbf{a}_{i,\mu} \right)  = I \otimes \dots \otimes I \otimes \diag\left( \mathrm{a}_{i,\mu} \right) \otimes I \otimes \dots \otimes I,
\end{equation*}
and
\begin{equation*}
    \diag\left( \mathbf{a}_{i,i+1,\mu} \right)  = \sum_{k=1}^{r_{i,i+1,\mu}}I \otimes \dots \otimes I \otimes \diag\left( \left( \mathrm{a}_{i,i+1,\mu}^{(1)} \right)_{k,:} \right) \otimes \diag\left( \left( \mathrm{a}_{i,i+1,\mu}^{(2)} \right)_{k,:} \right) \otimes I \otimes \dots \otimes I,
\end{equation*}
respectively. Analogously, for a cyclic system we obtain
\begin{equation*}
    \diag\left( \mathbf{a}_{d,1,\mu} \right)  = \sum_{k=1}^{r_{d,1,\mu}} \diag\left( \left( \mathrm{a}_{d,1,\mu}^{(2)} \right)_{k,:} \right) \otimes I \otimes \dots \otimes I \otimes \diag\left( \left( \mathrm{a}_{d,1,\mu}^{(1)} \right)_{k,:} \right).
\end{equation*}
Furthermore, any elementary reaction only changes the configuration of one or two cells. Thus, any vector of net changes has the form
\begin{equation} \label{single-cell net change}
    \xi_{i,\mu} = (0, \dots, 0, p_{i,\mu}, 0, \dots, 0)^T,
\end{equation}
with $p_{i,\mu} \in \mathbb{Z}$, for an SCR $\mathrm{R}_{i,\mu}$, or 
\begin{equation} \label{two-cell net change}
    \xi_{i,i+1,\mu} = (0, \dots, 0, p_{i,i+1,\mu}, q_{i,i+1,\mu}, 0, \dots, 0)^T,
\end{equation}
with $p_{i,i+1,\mu},q_{i,i+1,\mu} \in \mathbb{Z}$, for a TCR $\mathrm{R}_{i,i+1,\mu}$. As as result, the multidimensional shift operators also have a special structure. Using $ G_i(0) = I \in \R^{n_i \times n_i}$, we obtain a shift operator $\mathbf{G}_{i,\mu}$ belonging to an SCR $\textrm{R}_{i,\mu}$
\begin{equation*}
    \mathbf{G}_{i,\mu} = I \otimes \dots \otimes I \otimes G_i(-p_{i,\mu}) \otimes I \otimes \dots \otimes I
\end{equation*}
and for a shift operator $\mathbf{G}_{i,i+1,\mu}$ belonging to a TCR $\textrm{R}_{i,i+1,\mu}$ (we identify $\textrm{R}_{d,d+1,\mu} $ with $\textrm{R}_{d,1,\mu}$)
\begin{equation*}
    \mathbf{G}_{i,i+1,\mu} = I \otimes \dots \otimes I \otimes G_i(-p_{i,i+1,\mu} ) \otimes G_{i+1}(-q_{i,i+1,\mu} ) \otimes I \otimes \dots \otimes I,
\end{equation*}
and
\begin{equation*}
    \mathbf{G}_{d,1,\mu} = G_{1}(-q_{d,1,\mu}) \otimes I \otimes \dots \otimes I \otimes G_d(-p_{d,1,\mu}).
\end{equation*}
That is, only one or two shift matrices unequal to an identity matrix appear within these multidimensional shift operators. The properties above imply that we can write the operator $\mathbf{A}$ of a non-cyclic NNIS as
\begin{equation} \label{eq: MME operator splitted}
    \mathbf{A} = \sum_{i=1}^{d} \sum_{\mu=1}^{\alpha_i} \mathbf{A}_{i,\mu} + \sum_{i=1}^{d-1} \sum_{\mu=1}^{\beta_i} \mathbf{A}_{i,i+1,\mu},
\end{equation}
with
\begin{equation} \label{eq: component}
\begin{split}
    \mathbf{A}_{i,\mu} & = \left( \mathbf{G}_{i,\mu} - \mathbf{I} \right)\cdot \diag\left( \mathbf{a}_{i,\mu}\right) \\
    &= I \otimes \dots \otimes I \otimes \left( G_i(-p_{i,\mu} ) \cdot \diag\left(\mathrm{a}_{i,\mu}\right)\right) \otimes I \otimes \dots \otimes I \\
    & \quad - I \otimes \dots \otimes I \otimes  \diag\left(\mathrm{a}_{i,\mu}\right)  \otimes I \otimes \dots \otimes I
\end{split}
\end{equation}
and
\begin{equation} \label{eq: component 2}
\begin{split}
    \mathbf{A}_{i,i+1,\mu} & = \left( \mathbf{G}_{i,i+1,\mu} - \mathbf{I} \right)\cdot \diag\left( \mathbf{a}_{i,i+1,\mu}\right) \\ 
        &= \sum_{k=1}^{r_{i,i+1,\mu}} I \otimes \dots \otimes I \otimes \left(  G_i(-p_{i,i+1,\mu})  \cdot \diag\left(\left(\mathrm{a}_{i,i+1,\mu}^{(1)}\right)_{k,:}\right)\right) \\
        & \qquad \qquad \otimes \left(  G_{i+1}(-q_{i,i+1,\mu})  \cdot \diag\left(\left(\mathbf{a}_{i,i+1,\mu}^{(2)}\right)_{k,:}\right)\right) \otimes I \otimes \dots \otimes I\\
    & \qquad - \sum_{k=1}^{r_{i,i+1,\mu}} I \otimes \dots \otimes I \otimes  \diag\left(\left(\mathrm{a}_{i,i+1,\mu}^{(1)}\right)_{k,:}\right) \otimes  \diag\left(\left(\mathbf{a}_{i,i+1,\mu}^{(2)}\right)_{k,:}\right) \otimes I \otimes \dots \otimes I.
\end{split}
\end{equation}
For a cyclic system, we add the term $\sum_{\mu=1}^{\beta_d} \mathbf{A}_{d,1,\mu}$ to (\ref{eq: MME operator splitted}), with $\mathbf{A}_{d,1,\mu} = \left( \mathbf{G}_{d,1,\mu} - \mathbf{I} \right) \cdot \diag\left(\mathbf{a}_{d,1,\mu}\right)$. Note that~\eqref{eq: component} holds for all elementary reactions taking place only on the cell $\Theta_i$, whereas reactions involving two adjacent cells $\Theta_i$ and $\Theta_{i+1}$ can be represented by \eqref{eq: component 2}. In order to simplify the notation, we now define the matrices
\begin{equation}
    S_{i,\mu} = G_i(-p_{i,\mu} ) \cdot \diag\left(\mathrm{a}_{i,\mu}\right) \qquad \textrm{and} \qquad \tilde{S}_{i,\mu} = \diag\left(\mathrm{a}_{i,\mu}\right),
\end{equation}
for $i=1, \dots, d$ and $\mu=1, \dots, \alpha_i$, as well as
\begin{equation} \label{eq: SLIM matrices - L,M}
    \begin{aligned} 
        L_{i,\mu,k} & = G_i(-p_{i,i+1,\mu})  \cdot \diag\left(\left(\mathrm{a}_{i,i+1,\mu}^{(1)}\right)_{k,:}\right), & \tilde{L}_{i,\mu,k} & = \diag\left(\left(\mathrm{a}_{i,i+1,\mu}^{(1)}\right)_{k,:}\right),\\
        M_{i+1,\mu,k} & = G_{i+1}(-q_{i,i+1,\mu})  \cdot \diag\left(\left(\mathrm{a}_{i,i+1,\mu}^{(2)}\right)_{k,:}\right), & \quad \tilde{M}_{i+1,\mu,k} & = \diag\left(\left(\mathrm{a}_{i,i+1,\mu}^{(2)}\right)_{k,:}\right),
    \end{aligned}
\end{equation}
for $ i = 1, \dots, d-1$, $ \mu = 1, \dots, \beta_i $ and $k = 1, \dots, r_{i,i+1,\mu}$. For cyclic systems, we further define
\begin{equation*}
    \begin{aligned}
        L_{d,\mu,k} & = G_d(-p_{d,1,\mu})  \cdot \diag\left(\left(\mathrm{a}_{d,1,\mu}^{(1)}\right)_{k,:}\right), & \tilde{L}_{d,\mu,k} & = \diag\left(\left(\mathrm{a}_{d,1,\mu}^{(1)}\right)_{k,:}\right),\\
        M_{1,\mu,k} & = G_{1}(-q_{d,1,\mu})  \cdot \diag\left(\left(\mathrm{a}_{d,1,\mu}^{(2)}\right)_{k,:}\right), & \quad \tilde{M}_{1,\mu,k} & = \diag\left(\left(\mathrm{a}_{d,1,\mu}^{(2)}\right)_{k,:}\right),
    \end{aligned}
\end{equation*}
for $\mu = 1, \dots, \beta_d$ and $k = 1, \dots r_{d,1,\mu}$. Due to the bilinearity of the tensor product (see Appendix~\ref{app: Tensor Notation and Properties}), we can compute the sum of all matrices $S_{i,\mu}$ and $\tilde{S}_{i,\mu}$ and define 
\begin{equation} \label{eq: SLIM cores - S}
  \mathbf{S}_i = \sum_{\mu=1}^{\alpha_i} \left( S_{i,\mu} -  \tilde{S}_{i,\mu} \right).
\end{equation}
All SCRs on $\Theta_i$ can then be represented as $I \otimes \dots \otimes I \otimes \mathbf{S}_i \otimes I \otimes \dots \otimes I$. Written in the canonical tensor format, $\mathbf{A}$ has now the form
\begin{equation}\label{eq: preSLIM}
\begin{split}
    \mathbf{A} &= \mathbf{S}_1 \otimes I \otimes \dots \otimes I  + \dots + I \otimes \dots \otimes I \otimes \mathbf{S}_d \\
    & \quad + \sum_{\mu=1}^{\beta_1}  \sum_{k=1}^{r_{1,2,\mu}} \Big( L_{1,\mu,k} \otimes M_{2,\mu,k} \otimes I \otimes \dots \otimes I \Big) - \left( \tilde{L}_{1,\mu,k} \otimes \tilde{M}_{2,\mu,k} \otimes I \otimes \dots \otimes I \right) \\
    & \quad + \dots \\
    & \quad + \sum_{\mu=1}^{\beta_{d-1}} \sum_{k=1}^{r_{d-1,d,\mu}} \Big( I \otimes \dots \otimes I \otimes L_{d-1,\mu,k} \otimes M_{d,\mu,k} \big) - \left( I \otimes \dots \otimes I \otimes  \tilde{L}_{d-1,\mu,k} \otimes \tilde{M}_{d,\mu,k} \right) \\
    & \quad + \sum_{\mu=1}^{\beta_{d}} \sum_{k=1}^{r_{d,1,\mu}} \Big( M_{1,\mu,k} \otimes I \otimes \dots \otimes I \otimes  L_{d,\mu,k} \Big) -  \left( \tilde{M}_{1,\mu,k} \otimes I \otimes \dots \otimes I \otimes  \tilde{L}_{d,\mu,k} \right).
\end{split}
\end{equation}
The last sum is only required if we consider a cyclic interaction system, i.e.~there are $\beta_d$ elementary reactions between the cell $\Theta_d$ and $\Theta_1$. Equation \eqref{eq: preSLIM} is of the same type as \eqref{eq: SLIMTT op}. Thus, we can gather all the matrices $L_{i,\mu,k}, \tilde{L}_{i,\mu,k}$ and $M_{i+1,\mu,k}, \tilde{M}_{i+1,\mu,k}$ in the TT cores $\mathbf{L}_i$ and $\mathbf{M}_{i+1}$, respectively. The cores are then defined as
\begin{equation} \label{eq: two-cell cores 1}
    \begin{split}
        [ \mathbf{L}_i ] 
            & =
            \underbrace{\begin{bmatrix} L_{i,1,1} & -\tilde{L}_{i,1,1} & \dots & L_{i,\beta_i,r_{i,i+1,\beta_i}} &  -\tilde{L}_{i,\beta_i,r_{i,i+1,\beta_i}} \end{bmatrix}}_{\in \R^{1 \times n_i \times n_i \times (\beta_i \cdot r_{i,i+1,\beta_i})}}, \\
        [ \mathbf{M}_{i+1} ]
            & =
            \underbrace{\begin{bmatrix} M_{i+1,1,1} & \tilde{M}_{i+1,1,1} & \dots & M_{i+1,\beta_i,r_{i,i+1,\beta_i}} &  \tilde{M}_{i+1,\beta_i,r_{i,i+1,\beta_i}} \end{bmatrix}^\mathbb{T}}_{\in \R^{(\beta_i \cdot r_{i,i+1,\beta_i}) \times n_{i+1} \times n_{i+1} \times 1}},
    \end{split}
\end{equation}
for $ i = 1, \dots, d-1 $, and
\begin{equation} \label{eq: two-cell cores 2}
    \begin{split}
        [ \mathbf{L}_d ]
            & =
            \underbrace{\begin{bmatrix} L_{d,1,1} & -\tilde{L}_{d,1,1} & \dots & L_{d,\beta_d,r_{d,1,\beta_d}} &  -\tilde{L}_{d,\beta_d,r_{d,1,\beta_d}} \end{bmatrix}^\mathbb{T}}_{\in \R^{( \beta_d \cdot r_{d,1,\beta_d} )\times n_d \times n_d \times 1}}, \\
        [ \mathbf{M}_1 ]
            & =
            \underbrace{\begin{bmatrix} M_{1,1,1} & \tilde{M}_{1,1,1} & \dots & M_{1,\beta_d,r_{d,1,\beta_d}} &  \tilde{M}_{1,\beta_d,r_{d,1,\beta_d}} \end{bmatrix}}_{\in \R^{1 \times n_1 \times n_1 \times (\beta_d \cdot r_{d,1,\beta_d})}}.
    \end{split}
\end{equation}
These TT cores can now be inserted into Equation \eqref{eq: SLIMTT cyclic, heterogeneous} and we obtain the SLIM decomposition of the generator $\mathbf{A}$. Note that it may be possible to compress the cores $\mathbf{L}_i$ and $\mathbf{M}_{i+1}$. By considering the tensor product $\left[ \mathbf{L}_i \right] \otimes \left[ \mathbf{M}_{i+1} \right] $, we can conclude that only a basis of matrices of the core $\mathbf{L}_i$ as well as a basis of matrices of the core $\mathbf{M}_{i+1}$ is needed such that the tensor multiplication of these bases yields the same result as $\left[ \mathbf{L}_i \right] \otimes \left[ \mathbf{M}_{i+1} \right] $. This is explained in detail in Algorithm \ref{alg: compress}, where we use the multi-index notation described in Appendix~\ref{app: Tensor Notation and Properties}.

\begin{algorithm}[htb]
    \caption{Compression of 2-dimensional tensor-train operators}
    \label{alg: compress}
    \begin{algorithmic}[1]
        \State INPUT: Tensor train $\mathbf{T} = [\mathbf{T}_1] \otimes [\mathbf{T}_2]\in \R^{m \times m \times n \times n}$ with TT cores $[\mathbf{T}_1] = [\mathbf{T}_1^{(1)} \dots \mathbf{T}_1^{(\beta)}] \in \R^{1 \times m \times m \times \beta}$ and
        \LineCommentCont{$[\mathbf{T}_2] = [\mathbf{T}_2^{(1)} \dots \mathbf{T}_2^{(\beta)}]^T \in \R^{\beta \times n \times n \times 1}$.}
                \State Compute full tensor $\mathbf{T}$ and reshape it as a matrix $T \in \R^{(m \cdot m) \times (n \cdot n)}$.
                \State Apply compact singular value decomposition, i.e.\ $T = U \Sigma V^T$ with $U \in \R^{(m \cdot m) \times \gamma}$, $\Sigma \in \R^{\gamma \times \gamma}$, and $V \in \R^{(n \cdot n ) \times \gamma}$.
                \For{$k=1, \dots ,\gamma$}
                    \State Define $\tilde{\mathbf{T}}_1^{(k)} \in \R^{m \times m}$ by $ \left( \tilde{\mathbf{T}}_1^{(k)} \right)_{x,y} = U_{\overline{x,y},k}$.
                    \State Define $\tilde{\mathbf{T}}_2^{(k)} \in \R^{n \times n}$ by $ \left( \tilde{\mathbf{T}}_2^{(k)} \right)_{x,y} = (\Sigma V^T)_{k,\overline{x,y}}$.
                \EndFor
                \State OUTPUT: Tensor train $\tilde{\mathbf{T}} = [\tilde{\mathbf{T}}_1] \otimes [\tilde{\mathbf{T}}_2]$ with TT rank $\gamma \leq \beta$.
    \end{algorithmic}
\end{algorithm}

The cores corresponding to cyclic reactions between the cells $\Theta_d$ and $\Theta_1$ can also be compressed by applying Algorithm \ref{alg: compress} to $[\mathbf{L}_d]^\mathbb{T} \otimes [\mathbf{M}_1]^\mathbb{T}$. This becomes clear using the same argument as for \eqref{eq: propensity cyclic - 1} and \eqref{eq: propensity cyclic - 2}. The core components defined in \eqref{eq: SLIM cores - S}, \eqref{eq: two-cell cores 1}, and \eqref{eq: two-cell cores 2} (optionally after application of Algorithm \ref{alg: compress}) then form the corresponding elements of the SLIM decomposition given in \eqref{eq: SLIMTT cyclic, heterogeneous} and \eqref{eq: SLIMTT non-cyclic, heterogeneous}, respectively. Algorithm~\ref{alg: SLIM} can be used to automatically construct SLIM decompositions of master equation operators corresponding to systems based on nearest-neighbor interactions.

\begin{algorithm}[htbp]
    \caption{Construct SLIM decomposition of master equation operator}
    \label{alg: SLIM}
    \begin{algorithmic}[1]
        \State Given an NNIS with $d$ cells on state space $\mathcal{S} = \{1 , \dots , n_1\} \times \dots \times \{1 , \dots , n_d\}$. If the NNIS is cyclic, set $\Theta_{d+1} = \Theta_1$, $n_{d+1}=n_1$, $\textrm{R}_{d,d+1,\mu} = \textrm{R}_{d,1,\mu}$ etc.
        \State INPUT: \emph{Single-cell reactions} (SCR) 
        \LineCommentCont{For each cell $\Theta_i$, $1 \leq i \leq d$, and every $\textrm{R}_{i,\mu}$, $\mu = 1 , \dots , \alpha_i$, define the net change $p_{i, \mu} \in \mathbb{Z}$ (see (\ref{single-cell net change})) and the vector $\mathrm{a}_{i, \mu} \in \R^{n_i}$ (see (\ref{single-cell propensity})) containing the values of the corresponding reaction propensity.}
        \State \phantom{INPUT:} \emph{Two-cell reactions} (TCR)
        \LineCommentCont{For each pair of cells $\Theta_i$, $\Theta_{i+1}$, $1 \leq i \leq d-1$ ($1 \leq i \leq d$ if cyclic), and every $\textrm{R}_{i,i+1,\mu}$, $\mu = 1 , \dots , \beta_i$, define the net changes $p_{i,i+1, \mu}, q_{i,i+1,\mu} \in \mathbb{Z}$ (see \eqref{two-cell net change}) and the matrix $\mathrm{a}_{i, i+1,\mu} \in \R^{n_i \times n_{i+1}}$ (see \eqref{two-cell propensity}) containing the values of the corresponding reaction propensity.}
        \For{$i=1, \dots ,d$}
        \State Compute $\mathbf{S}_i = \sum_{\mu=1}^{\alpha_i} \left( G_i(-p_{i,\mu}) - I\right) \cdot \diag \left({\mathrm{a}_{i,\mu}} \right)$ as defined in (\ref{eq: SLIM cores - S}).
    \EndFor
    \For{$i=1, \dots ,d-1$ ($i=1, \dots , d$ if NNIS is cyclic)}
      \For{$\mu=1, \dots ,\beta_i$}
        \State Compute canonical representation of $\mathrm{a}_{i,i+1,\mu}$, i.e.\ $\mathrm{a}_{i,i+1,\mu} = \sum_{k=1}^{r_{i,i+1,\mu}} \left( \mathrm{a}^{(1)}_{i,i+1,\mu} \right)_{k,:} \otimes \left( \mathrm{a}^{(2)}_{i,i+1,\mu} \right)_{k,:}$.
        \State Compute $L_{i,\mu,k}$, $\tilde{L}_{i,\mu,k}$, $M_{i+1,\mu,k}$, and  $\tilde{M}_{i+1,\mu,k}$ as defined in (\ref{eq: SLIM matrices - L,M}).
      \EndFor
      \State Construct $\mathbf{L}_i$ and $\mathbf{M}_{i+1}$ as defined in (\ref{eq: two-cell cores 1}) and (\ref{eq: two-cell cores 2}).
      \State Apply Algorithm \ref{alg: compress} to $[\mathbf{L}_i] \otimes [\mathbf{M}_{i+1}]$ in order to compress the cores $\mathbf{L}_i$ and $\mathbf{M}_i$.
    \EndFor
    \State OUTPUT: SLIM decomposition of master equation operator $\mathbf{A}$ as given in (\ref{eq: SLIMTT cyclic, heterogeneous}) and (\ref{eq: SLIMTT non-cyclic, heterogeneous}), respectively.
    \end{algorithmic}
\end{algorithm}

\begin{example}
  Using the decompositions given in Example~\ref{ex: cascade 2}, we now construct the SLIM decomposition corresponding to our signal cascade model with $d$ genes, cf.~\cite{DOLGOV01}. The reaction propensities satisfy \eqref{single-cell propensity}, i.e.~they can be written as a rank-one tensor where only one component is unequal to a vector of ones. As defined in \eqref{equation:TTME:operator}, the master equation operator is given by
  \begin{equation*}
    \mathbf{A} = \sum_{\mu=1}^{M} (\mathbf{G}_{\mu} - \mathbf{I} ) \cdot \diag(\textbf{a}_{\mu}),
  \end{equation*}
  with $M = 2d$. In canonical format, we can express this as
  \begin{equation*}
    \begin{split}
      \mathbf{A} & =  0.7 \cdot G^{\downarrow} \otimes I \otimes \dots \otimes I ~~ - ~~ 0.7 \cdot I \otimes I \otimes \dots \otimes I \\
      & ~~~ + ~ H_1 \otimes G^{\downarrow} \otimes I \otimes \dots \otimes I ~~ - ~~ H_1 \otimes I \otimes \dots \otimes I \\
      & ~~~ + ~ \dots \\
      & ~~~ + ~ I \otimes \dots \otimes I \otimes H_1 \otimes G^{\downarrow} ~~ - ~~ I \otimes \dots \otimes I \otimes H_1 \otimes I \\
      & ~~~ + ~ ( G^{\uparrow} \cdot H_2) \otimes I \otimes \dots \otimes I ~~ - H_2 \otimes I \otimes \dots \otimes I \\
      & ~~~ + ~ \dots \\
      & ~~~ + ~ I \otimes \dots \otimes I \otimes ( G^{\uparrow} \cdot H_2) ~~ - ~~ I \otimes \dots \otimes I \otimes H_2,
    \end{split}
  \end{equation*}
  with identity matrix $I \in \R^{64 \times 64}$, $G^{\downarrow}$ and $G^{\uparrow}$ as defined in Example \ref{ex: cascade 2} and
  \begin{equation*}
    H_1 = \diag\left(\frac{0}{5}, \frac{1}{6}, \dots, \frac{63}{68}\right), \qquad H_2 = 0.07 \cdot \diag(0, 1, \dots , 63),
  \end{equation*}
  where $\diag(v)$ denotes the diagonalization of the vector $v \in \R^{64}$. By defining 
  \begin{equation*}
    \mathbf{S}^{*} = 0.7 \cdot \left( G^{\downarrow} -  I \right) , \quad \mathbf{S} = \left( G^{\uparrow} - I \right) \cdot H_2 , \quad  \mathbf{L}=H_1, \quad  \mathbf{I}=I, \quad \mathbf{M}=G^{\downarrow} - I,
  \end{equation*}
  we obtain the non-cyclic SLIM decomposition
  \begin{equation}
      \mathbf{A} =
      \begin{bmatrix}
      \mathbf{S}^{*} & \mathbf{L} & \mathbf{I} 
      \end{bmatrix}
      \otimes 
      \begin{bmatrix}
      \mathbf{I} & 0          & 0                     \\
      \mathbf{M} & 0          & 0                     \\       
      \mathbf{S} & \mathbf{L} & \mathbf{I}            
      \end{bmatrix}
      \otimes \dots \otimes
      \begin{bmatrix}
      \mathbf{I} & 0          & 0                     \\
      \mathbf{M} & 0          & 0                     \\       
      \mathbf{S} & \mathbf{L} & \mathbf{I}            
      \end{bmatrix}
      \otimes
      \begin{bmatrix}
      \mathbf{I} \\
      \mathbf{M} \\       
      \mathbf{S} 
      \end{bmatrix}. \tag*{\exampleSymbol}
  \end{equation} 
\end{example}

\begin{remark}
In order to speed up calculations and to reduce the storage consumption even further, one can apply the so-called \emph{quantized tensor-train format} (QTT format) \cite{Ose09, Ose10, KHOROMSKIJ2011}. That is, we reorder the elements of each core in a new tensor which has the same number of elements but a higher order and smaller mode sizes. These tensors can then be split into several QTT cores, e.g.~by applying singular value decompositions. In \cite{DOLGOV01}, a TT/QTT decomposition of the operator $\mathbf{A}$ and some numerical results using the Crank-Nicolson time integration scheme can be found. 
\end{remark}

\subsection{Examples}

\subsubsection{CO Oxidation at Ru\texorpdfstring{O$_2$}{O2}}

The first example is a cyclic, homogeneous NNIS that we already considered in \cite{GELSS2016}. However, the SLIM decomposition, which is more general, had not been developed at that time and a different TT decomposition was used instead. Here, we consider a heterogeneous catalytic process where the cells $\Theta_1, \dots, \Theta_d$ represent adsorption sites on a $\textrm{RuO}_2$(110) surface, see Figure \ref{fig: RuO2 a}. The aim is to simulate the CO oxidation at the surface. Because it has been found that the chemical kinetics predominantly take place only on the \emph{coordinatively unsaturated sites} (cus) \cite{MESKINE}, we construct a ring of $d$ cus sites, see Figure \ref{fig: RuO2 b} and \ref{fig: RuO2 c}. 

\begin{figure}[H]
    \centering
    \begin{subfigure}[b]{0.35\textwidth}
        \centering\caption{}
        \includegraphics[height=85px]{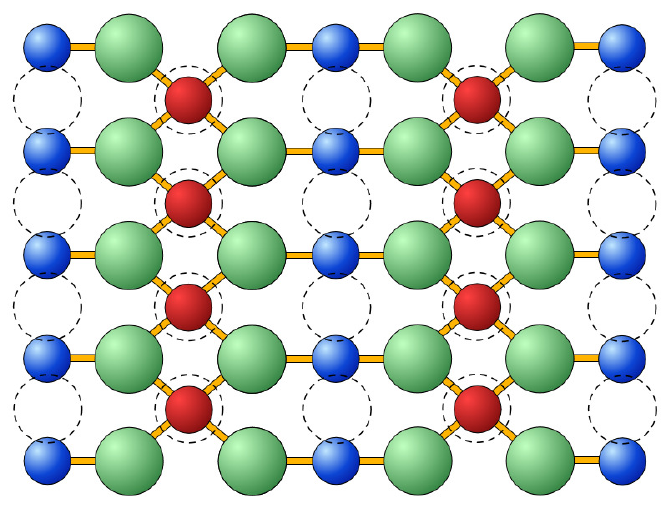}
        \label{fig: RuO2 a}
    \end{subfigure}
    \hspace{2mm}
    \begin{subfigure}[b]{0.35\textwidth}
        \centering \caption{}
        \includegraphics[height=85px]{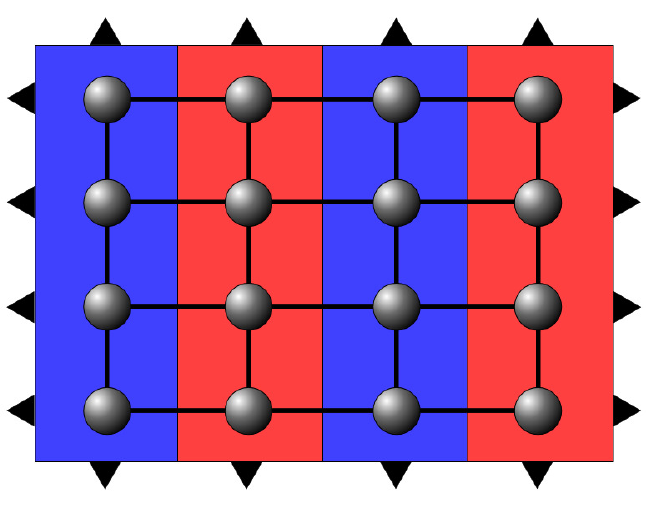}
        \label{fig: RuO2 b}
    \end{subfigure}
    \hspace{2mm}
    \begin{subfigure}[b]{0.15\textwidth}
        \centering \caption{}
        \includegraphics[height=85px]{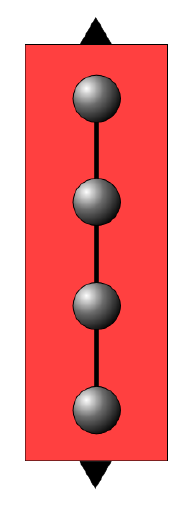}
        \label{fig: RuO2 c}
    \end{subfigure}
    \caption{(a) Top view of the $\textrm{RuO}_2$(110) surface showing the two prominent adsorption sites, bridge sites between the ruthenium atoms in blue and cus sites on the ruthenium atoms in red. (b) 2D lattice model of the coarse-grained surface composed of alternating rows of bridge and cus sites as used in \cite{REUTER, MATERA}. (c) 1D lattice model completely composed of cus sites.}
    \label{fig: RuO2}
\end{figure}

Each site may be in three different states (1 $\widehat{=}$ empty, 2 $\widehat{=}$ O-covered, 3 $\widehat{=}$ CO-covered). The possible events are unimolecular adsorption/desorption of CO, dissociative oxygen adsorption on two neighboring sites and the corresponding reverse processes, diffusion of adsorbed CO/O to a neighboring site, and the formation of gaseous CO$_2$ from adsorbed CO and O on neighboring sites. For further details on the established microkinetic model for the CO oxidation at $\textrm{RuO}_2$(110) we refer to \cite{REUTER}. Table \ref{table:reactions} summarizes the elementary reactions of the reduced model and the specific values of the reaction propensities. For a detailed description of the gas phase conditions, see e.g.~\cite{MATERA}. 

\begin{table}[H]
\centering
\renewcommand{\arraystretch}{1.2}
\setlength\tabcolsep{1.5mm}
\begin{footnotesize}
\begin{tabular}{|llcccll|}
    \hline
    \multicolumn{7}{|>{\columncolor[gray]{0.9}}l|}{\textbf{Adsorption}} \\ \hline \hline
$\textrm{R}^{\textrm{Ad}}_{\textrm{O}_2}$ &:&    $\varnothing_i + \varnothing_{j}$ & $\rightarrow$ & $\textrm{O}_i + \textrm{O}_{j}$ & , \quad $k^{\textrm{Ad}}_{\textrm{O}_2}$ &$=~~~9.7 \cdot 10^7 s^{-1}$ \\
$\textrm{R}^{\textrm{Ad}}_{\textrm{CO}}$ &:&    $\varnothing_i$ & $\rightarrow$ & $\textrm{CO}_i$ & , \quad $k^{\textrm{Ad}}_{\textrm{CO}}$ &$=~~~10^4 - 10^{10} s^{-1}$ \\
\hline
\hline
\multicolumn{7}{|>{\columncolor[gray]{0.9}}l|}{\textbf{Desorption}} \\ \hline \hline
$\textrm{R}^{\textrm{De}}_{\textrm{O}_2}$ &:&    $\textrm{O}_i + \textrm{O}_{j}$ & $\rightarrow$ & $\varnothing_i + \varnothing_{j}$ & , \quad  $k^{\textrm{De}}_{\textrm{O}_2}$ &$=~~~2.8 \cdot 10^1 s^{-1}$ \\
$\textrm{R}^{\textrm{De}}_{\textrm{CO}}$ &:&    $\textrm{CO}_i$ & $\rightarrow$ & $\varnothing_i$ & , \quad$k^{\textrm{De}}_{\textrm{CO}}$ &$=~~~9.2 \cdot  10^6 s^{-1}$ \\
$\textrm{R}^{\textrm{De}}_{\textrm{CO}_2}$ &:&    $\textrm{CO}_i + \textrm{O}_{j} $ & $\rightarrow$ & $\varnothing_i + \varnothing_{j}$ & , \quad $k^{\textrm{De}}_{\textrm{CO}_2}$ &$=~~~1.7 \cdot 10^5 s^{-1}$ \\
\hline
\hline
\multicolumn{7}{|>{\columncolor[gray]{0.9}}l|}{\textbf{Diffusion}} \\ \hline \hline
$\textrm{R}^{\textrm{Diff}}_{\textrm{O}}$ &:&    $\textrm{O}_i + \varnothing_{j}$ & $\rightarrow$ & $\varnothing_i + \textrm{O}_{j}$ & , \quad $k^{\textrm{Diff}}_{\textrm{O}}$ &$=~~~0.5 s^{-1}$ \\
$\textrm{R}^{\textrm{Diff}}_{\textrm{CO}}$ &:&    $\textrm{CO}_i + \varnothing_{j}$ & $\rightarrow$ & $\varnothing_i + \textrm{CO}_{j}$ & , \quad $k^{\textrm{Diff}}_{\textrm{CO}}$ &$=~~~6.6 \cdot 10^{-2} s^{-1}$ \\
\hline
\end{tabular}
\end{footnotesize}
\caption{Elementary reaction steps on the cus sites together with their corresponding rate constants, see~\cite{REUTER} for details. The reactions are defined on two neighboring sites $\Theta_i$ and $\Theta_{j}$, except for adsorption and desorption of CO, which are defined only on site $\Theta_i$.}
\label{table:reactions}
\end{table}

\noindent In order to construct the operator according to the MME, we use Algorithm \ref{alg: SLIM} with inputs
\begin{align*}
    \mathrm{a}_{i,1} & = \begin{pmatrix} k_{\textrm{CO}}^{\textrm{Ad}} & 0 & 0 \end{pmatrix}^T,  & p_{i,1} & = +2,\\
    \mathrm{a}_{i,2} & = \begin{pmatrix} 0 & 0 &  k_{\textrm{CO}}^{\textrm{De}} \end{pmatrix}^T, & p_{i,2} & = -2,
\end{align*}
and
\newcommand\w[1]{\makebox[0.8cm]{$#1$}}
\begin{align*}
    \mathrm{a}_{i,i+1,1} & = \begin{pmatrix} k_{\textrm{O}_2}^{\textrm{Ad}} & 0 & 0 \\ \w0 & \w0 & \w0 \\ 0 & 0 & 0 \end{pmatrix},  & [ p_{i,i+1,1}, \, q_{i,i+1,1}] & = [+1,\, +1],\\
    \mathrm{a}_{i,i+1,2} & = \begin{pmatrix} \w0 & \w0 & \w0 \\ 0 & k_{\textrm{O}_2}^{\textrm{De}} & 0  \\ 0 & 0 & 0 \end{pmatrix},  & [ p_{i,i+1,2}, \, q_{i,i+1,2}] & = [-1,\, -1],\\
    \mathrm{a}_{i,i+1,3} & = \begin{pmatrix} \w0 & \w0 & \w0 \\ 0 & 0 & 0 \\ 0 & k_{\textrm{CO}_2}^{\textrm{De}} & 0\end{pmatrix},  & [ p_{i,i+1,3}, \, q_{i,i+1,3}] & = [-2,\, -1],\\
    \mathrm{a}_{i,i+1,4} & = \begin{pmatrix} \w0 & \w0 & \w0 \\ 0 & 0 & k_{\textrm{CO}_2}^{\textrm{De}}  \\ 0 & 0 & 0 \end{pmatrix},  & [ p_{i,i+1,4}, \, q_{i,i+1,4}] & = [-1,\, -2],\\
    \mathrm{a}_{i,i+1,5} & = \begin{pmatrix} \w0 & \w0 & \w0 \\ k_{\textrm{O}}^{\textrm{Diff}} & 0 & 0 \\ 0 & 0 & 0 \end{pmatrix},  & [ p_{i,i+1,5}, \, q_{i,i+1,5}] & = [-1,\, +1],\\
    \mathrm{a}_{i,i+1,6} & = \begin{pmatrix} 0 & k_{\textrm{O}}^{\textrm{Diff}} & 0 \\ \w0 & \w0 & \w0 \\  0 & 0 & 0 \end{pmatrix},  & [ p_{i,i+1,6}, \, q_{i,i+1,6}] & = [+1,\, -1],\\
    \mathrm{a}_{i,i+1,7} & = \begin{pmatrix} \w0 & \w0 & \w0 \\  0 & 0 & 0 \\ k_{\textrm{CO}}^{\textrm{Diff}} & 0 & 0 \end{pmatrix},  & [ p_{i,i+1,7}, \, q_{i,i+1,7}] & = [-2,\, +2],\\
    \mathrm{a}_{i,i+1,8} & = \begin{pmatrix} 0 & 0 & k_{\textrm{CO}}^{\textrm{Diff}} \\ \w0 & \w0 & \w0 \\  0 & 0 & 0  \end{pmatrix},  & [ p_{i,i+1,8}, \, q_{i,i+1,8}] & = [+2,\, -2].
\end{align*}
Since we consider a cyclic, homogeneous NNIS here, the inputs above hold for $i=1, \dots , d$, where $\mathrm{a}_{d,d+1,\mu} = \mathrm{a}_{d,1,\mu}$ and $[ p_{d,d+1,\mu}, \, q_{d,d+1,\mu}] = [ p_{d,1,\mu}, \, q_{d,1,\mu}]$. The output of Algorithm \eqref{alg: SLIM} is then a TT operator with TT ranks equal to 16, which is the same size as the operator in \cite{GELSS2016}. Using this exact tensor-train decomposition, we can compute stationary and time dependent probability distributions by formulating eigenvalue problems or applying implicit time propagation schemes combined with ALS. In \cite{GELSS2016}, we carried out several experiments including the analysis of the computational costs for an increasing number of dimensions, computing central quantities describing the efficiency of the catalyst, and a demonstration of the advantage of the TT approach for stiff systems.  

\subsubsection{Toll Station}

As a final example, we examine a quasi-realistic traffic problem. Imagine a toll station with $d$ lanes. Cars form a queue in these lanes arriving according to a given distribution. Each car can then change its lane but may only go from one lane to a neighboring one, depending on a given interaction parameter. We assume that the time to pass the toll station depends on the toll booths. In our example, we choose a normal distribution for the incoming cars and a sum of two normal distributions for the outgoing flux, see Figure~\ref{fig: Toll a}.

The state space is given by
\begin{equation*}
  \mathcal{S} = \{1, \dots, n \} \times \dots \times \{1, \dots , n \},
\end{equation*}
and state $X = (x_1 , \dots, x_d)^T \in \mathcal{S}$ now represents the number of cars in each lane, e.g.~there are $x_i -1$ cars on lane $\Theta_i$. In what follows, we set $d = 20$, $n=10$, and define the functions
\begin{equation*}
  \begin{split}
    f_\mathrm{in}(t) & = \frac{1}{\sqrt{2 \pi \sigma_\mathrm{in}^2}} \cdot e^{-\frac{1}{2} \frac{t^2}{\sigma_\mathrm{in}^2} } + 0.05,\\
    f_\mathrm{out}(t) & = \frac{1}{\sqrt{2 \pi \sigma_\mathrm{out, left }^2}} \cdot e^{-\frac{1}{2} \frac{(t-\nu_\mathrm{out,left})^2}{\sigma_\mathrm{out,left}^2} } + \frac{1}{\sqrt{2 \pi \sigma_\mathrm{out, right }^2}} \cdot e^{-\frac{1}{2} \frac{(t-\nu_\mathrm{out,right})^2}{\sigma_\mathrm{out,right}^2} },
  \end{split}
\end{equation*}
with $\sigma_\mathrm{in}^2 = 2.5$, $\sigma_\mathrm{out,left}^2 = 1$, $\sigma_\mathrm{out,right}^2 = 0.5$, $\nu_\mathrm{out,left} = -1.5$, and $\nu_\mathrm{out,right} = 1.5$. The positions of the lanes are given by $t_i = -2+0.5 (i-1)$, for $i=1, \dots , d$. For the rate to change from lane $\Theta_i$ to lane $\Theta_{i+1}$ and vice versa, we assume a step function on $x_i - x_{i+1}$ which is only unequal to zero if there are fewer cars in the neighboring lane, i.e. 
\begin{equation*}
  f_\mathrm{change}(x_i - x_{i+1}) = 
  \begin{cases}
  5, \quad \textrm{if} \quad x_i - x_{i+1} > 0,\\
  0, \quad \textrm{otherwise}. 
  \end{cases}
\end{equation*}

\begin{figure}[htb]
    \centering
    \includegraphics[width=0.55\textwidth]{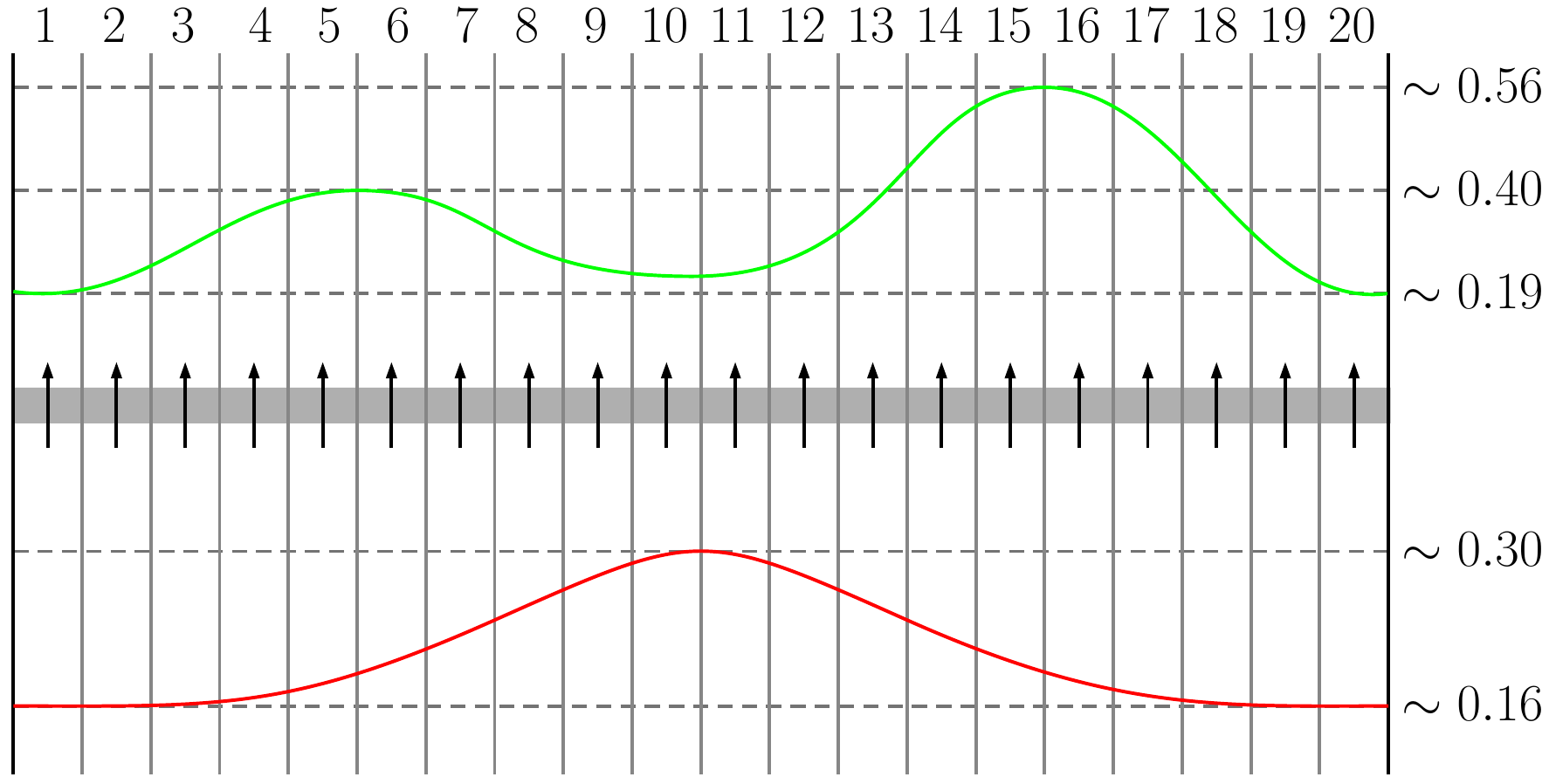}
    \caption{Pictorial representation of the toll station model. The red curve shows the probability distribution of incoming cars, the green curve the probability distribution of processing times.}
    \label{fig: Toll a}
\end{figure}

\noindent Again, we apply Algorithm \ref{alg: SLIM} to construct the generator according to this heterogeneous, non-cyclic NNIS. The inputs are
\begin{align*}
    \mathrm{a}_{i,1} & = f_\mathrm{in}(t_i) \cdot \begin{pmatrix} 1 & 1 & \dots & 1 \end{pmatrix}^T,  & p_{i,1} & = +1,\\
    \mathrm{a}_{i,2} & = f_\mathrm{out}(t_i) \cdot \begin{pmatrix} 0 & 1 & \dots & 1 \end{pmatrix}^T, & p_{i,2} & = -1,
\end{align*}
and matrices $\mathrm{a}_{i,i+1,1}, \mathrm{a}_{i,i+1,2} \in \R^{n \times n}$ with $( \mathrm{a}_{i,i+1,1})_{x_i , x_{i+1}} = f_\mathrm{change}(x_i - x_{i+1})$ and $( \mathrm{a}_{i,i+1,2})_{x_i , x_{i+1}} = f_\mathrm{change}(x_{i+1} - x_{i})$. Furthermore, $ [ p_{i,i+1,1}, \, q_{i,i+1,1}]  = [-1,\, +1]$ and $ [ p_{i,i+1,2}, \, q_{i,i+1,2}]  = [+1,\, -1]$.

In this case, the output of Algorithm~\ref{alg: SLIM} is a tensor-train operator $\mathbf{A} \in \R^{(n \times n) \times \dots \times (n \times n)}$ with ranks equal to 39. We are interested in the transient behavior of the distribution of cars depending on a given initial state. For this purpose, we apply the implicit Euler method
\begin{equation}\label{eq: SLE}
  (\mathbf{I} - \tau \mathbf{A}) \, \mathbf{T}_{k+1} = \mathbf{T}_k,
\end{equation}
with step size $\tau = 10^{-1}$ and set $\mathbf{T}_0 \in \R^{n \times \dots \times n}$ to the singular probability distribution with $\left( \mathbf{T}_0 \right)_{6, \dots, 6} = 1$. The resulting systems of linear equations are solved with ALS where the TT ranks of the solution have been arbitrarily set to 10. The simulation results in Figure~\ref{fig: Toll b} show that the constant distribution at the beginning rapidly changes due to the different input and output rates of the lanes. Also, the length of the queues of cars are decreasing over a short time interval. In order to evaluate the accuracy of the results, we consider the relative errors of the systems of linear equations given in \eqref{eq: SLE}:
\begin{equation*}
\varepsilon_k = \frac{\left\| (\mathbf{I} - \tau \mathbf{A}) \mathbf{T}_{k} -  \mathbf{T}_{k-1} \right\|_F}{\left\| \mathbf{T}_{k-1} \right\|_F},
\end{equation*}
where $\left\|\,.\,\right\|_F$ denotes the Frobenius norm for tensors. The errors for all 300 steps are less than 5 \%.

\begin{figure}[htb]
    \centering
    \includegraphics[width=0.5\textwidth]{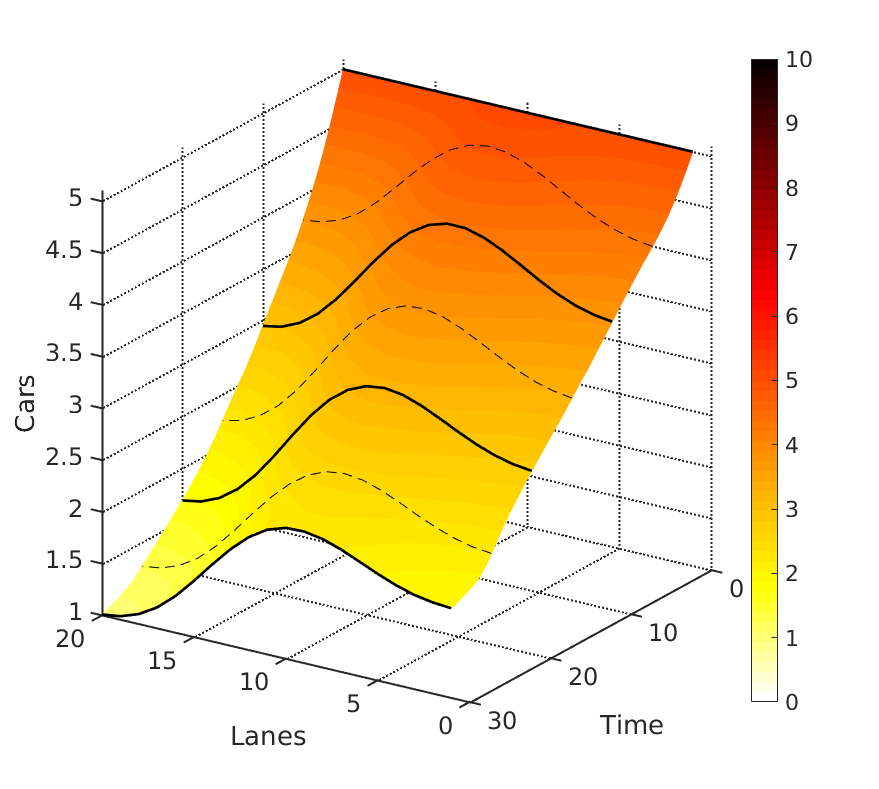}
    \caption{Simulation results for the toll station model.}
    \label{fig: Toll b}
\end{figure}

\section{Conclusion and Outlook}
\label{sec: concl}

In this paper, we proposed an approach to construct TT decompositions of potentially high-dimensional nearest-neighbor interaction systems. The aim is to reduce the memory consumption as well as the computational costs significantly and to mitigate the curse of dimensionality. First, we have shown how to apply the SLIM decomposition to general NNISs and then we gave a detailed description of TT decompositions of high-dimensional Markovian master equations. Additionally, we presented algorithms which can be used to construct SLIM decompositions of Markovian generators automatically. The results, which were illustrated with several examples from different application areas such as quantum physics and heterogeneous catalysis, show that by exploiting the coupling structure of a system it is possible to compute low-rank tensor decompositions of probability distributions and associated linear operators. We also considered homogeneous systems where the ranks of the SLIM decomposition do not depend on the network size, resulting in a linear growth of the storage consumption.

Future research will include the consideration of nearest-neighbor interaction systems from other scientific areas and the examination of more general interaction systems, especially the generalization of the SLIM decomposition to next-nearest-neighbor interaction systems or other systems with a certain coupling structure.

\section*{Acknowledgments}
This research has been funded by the Berlin Mathematical School, the Einstein Center for Mathematics, Deutsche Forschungsgemeinschaft through grant CRC 1114, and \textsc{Matheon}, which is supported by the Einstein Foundation Berlin.

\bibliographystyle{unsrt}
\bibliography{SLIM}

\pagebreak

\appendix

\section{Tensor Notation and Properties}
\label{app: Tensor Notation and Properties}

\subsection{Bilinearity of the Tensor Product}

The tensor product is a bilinear map, i.e.\ the following properties hold:
\begin{equation*}
\begin{alignedat}{2}
    & \text{(i)}   & \quad & \left( \mathbf{T}_1 + \mathbf{T}_2 \right) \otimes \mathbf{U} = \mathbf{T}_1 \otimes \mathbf{U} + \mathbf{T}_2 \otimes \mathbf{U}, \\
    & \text{(ii)}  & \quad & \mathbf{T} \otimes \left( \mathbf{U}_1 + \mathbf{U}_2 \right) = \mathbf{T} \otimes \mathbf{U}_1 + \mathbf{T} \otimes \mathbf{U}_2, \\
    & \text{(iii)} & \quad & \left( \lambda \cdot \mathbf{T} \right) \otimes \mathbf{U} = \mathbf{T} \otimes \left( \lambda \cdot \mathbf{U}_1\right) = \lambda \cdot \left( \mathbf{T} \otimes \mathbf{U} \right),
\end{alignedat}
\end{equation*}
for $\mathbf{T}, \mathbf{T}_1, \mathbf{T}_2 \in \mathbb{R}^{m_1 \times \ldots \times m_d}$, $\mathbf{U}, \mathbf{U}_1, \mathbf{U}_2 \in \R^{n_1 \times \ldots \times n_e}$, and $\lambda \in \R$.

\subsection{Rank-Transposed of a TT Core}
Given TT cores $\mathbf{T}^{(i)} \in \R^{r_{i-1} \times n_i \times r_i}$ and $\mathbf{A}^{(i)} \in \R^{s_{i-1} \times n_i \times n_i \times s_i}$ with 
\begin{equation*}
    \left[ \mathbf{T}^{(i)} \right] =
    \begin{bmatrix}
        & \mathbf{T}^{(i)}_{1,:,1} & \cdots & \mathbf{T}^{(i)}_{1,:,r_{i}} & \\
        & & & & \\
        & \vdots & \ddots & \vdots & \\
        & & & & \\
        & \mathbf{T}^{(i)}_{r_{i-1},:,1} & \cdots & \mathbf{T}^{(i)}_{r_{i-1},:,r_i} &
    \end{bmatrix},
    \quad \text{and} \quad
    \left[ \mathbf{A}^{(i)} \right] =
    \begin{bmatrix}
        & \mathbf{A}^{(i)}_{1,:,:,1} & \cdots & \mathbf{A}^{(i)}_{1,:,:,s_{i}} & \\
        & & & & \\
        & \vdots & \ddots & \vdots & \\
        & & & & \\
        & \mathbf{A}^{(i)}_{s_{i-1},:,:,1} & \cdots & \mathbf{A}^{(i)}_{s_{i-1},:,:,s_i} &
    \end{bmatrix},
\end{equation*}
we define the rank-transposed cores $\left[ \mathbf{T}^{(i)} \right]^\mathbb{T}$ and $\left[ \mathbf{A}^{(i)} \right]^\mathbb{T}$ as
\begin{equation*} 
    \left[ \mathbf{T}^{(i)} \right]^\mathbb{T} =
    \begin{bmatrix}
        & \mathbf{T}^{(i)}_{1,:,1} & \cdots & \mathbf{T}^{(i)}_{r_{i-1},:,:,1} & \\
        & & & & \\
        & \vdots & \ddots & \vdots & \\
        & & & & \\
        & \mathbf{T}^{(i)}_{1,:,:,r_i} & \cdots & \mathbf{T}^{(i)}_{r_{i-1},:,:,r_i} &
    \end{bmatrix},
    \quad \text{and} \quad
    \left[ \mathbf{A}^{(i)} \right]^\mathbb{T} =
    \begin{bmatrix}
        & \mathbf{A}^{(i)}_{1,:,:,1} & \cdots & \mathbf{A}^{(i)}_{s_{i-1},:,:,1} & \\
        & & & & \\
        & \vdots & \ddots & \vdots & \\
        & & & & \\
        & \mathbf{A}^{(i)}_{1,:,:,s_i} & \cdots & \mathbf{A}^{(i)}_{s_{i-1},:,:,s_i} &
    \end{bmatrix}.
\end{equation*}
Note that the vectors/matrices within the cores are not tranposed, only the outer indices of each element are interchanged.

\subsection{Equivalence of the Master Equation Formulations}

\begin{theorem*}
For any state $X = (x_1, \dots, x_d)^T \in \mathcal{S}$, it holds that
\begin{equation*}
    \left( \pd{}{t} \mathbf{P}(t) \right)_{x_1, \dots, x_d } = \pd{}{t} P(X,t).
\end{equation*}
\end{theorem*}

\begin{proof}
Following the definition of the tensor multiplication, see e.g.\ \cite{HACKBUSCH2012}, we can write
\begin{equation*}
  \begin{split}
    \left ( \pd{}{t} \mathbf{P}(t) \right)_{x_1, \dots, x_d }  & = \left ( \left ( \sum_{\mu=1}^{M} (\mathbf{G}_\mu - \mathbf{I} ) \cdot \diag(\mathbf{a}_\mu) \right ) \cdot \mathbf{P}(t) \right)_{x_1, \dots, x_d }\\
    &= \sum_{\mu=1}^M \sum_{i_1=1}^{n_1} \dots \sum_{i_d=1}^{n_d} \left( (\mathbf{G}_\mu - \mathbf{I}) \cdot \diag (\mathbf{a}_\mu)\right)_{x_1, i_1, \dots, x_d,i_d} \cdot \left(\mathbf{P}(t) \right)_{i_1, \dots, i_d}.
  \end{split}
\end{equation*}
Furthermore, it holds that
\begin{equation*}
\begin{split}
     \left( (\mathbf{G}_\mu - \mathbf{I}) \cdot \diag (\mathbf{a}_\mu)\right)_{x_1, i_1, \dots, x_d,i_d} & =  \sum_{j_1=1}^{n_1} \dots \sum_{j_d=1}^{n_d} \left( \mathbf{G}_\mu \right)_{x_1, j_1, \dots, x_d, j_d} \cdot \left( \diag(\mathbf{a}_\mu) \right)_{j_1, i_1, \dots, j_d, i_d} \\
    & \quad -\sum_{j_1=1}^{n_1} \dots \sum_{j_d=1}^{n_d} \left( \mathbf{I} \right)_{x_1, j_1, \dots, x_d,j_d} \cdot \left( \diag(\mathbf{a}_\mu) \right)_{j_1, i_1, \dots, j_d, i_d}.
\end{split}
\end{equation*}
Considering Definition \ref{def: shift operators} of the shift operators, this results in
\begin{equation*}
    \left( \diag(\mathbf{a}_\mu) \right)_{x_1 - \xi_{\mu}(1), i_1, \dots, x_d - \xi_{\mu}(d) , i_d} - \left( \diag(\mathbf{a}_\mu) \right)_{x_1, i_1, \dots, x_d, i_d}.
\end{equation*}
Just as $a_\mu (X) $ and $P(X,t)$ are set to zero if $X \notin \mathcal{S}$, we set 
\begin{equation*}
\left( \diag(\mathbf{a}_\mu) \right)_{x_1 - \xi_{\mu}(1) , i_1, \dots, x_d - \xi_{\mu}(d) , i_d} =0,
\end{equation*}
if $x_k - \xi_{\mu}(k)  \notin \{1, \dots , n_i \}$ for a $k \in \{1, \dots, d\}$. Analogously, we do the same for $\left( \mathbf{P} (t) \right)_{x_1 - \xi_{\mu}(1), \dots, x_d - \xi_{\mu}(d) }$. Due to the construction of $\diag(\mathbf{a}_\mu)$, we finally obtain
\begin{equation*}
    \left( \pd{}{t} \mathbf{P}(t) \right)_{x_1, \dots, x_d}
        = \sum_{\mu=1}^M a_\mu(X-\xi_\mu) P(X-\xi_\mu,t) - a_\mu(X) P(X,t)
        = \pd{}{t} P(X,t). \qedhere
\end{equation*}
\end{proof}

\subsection{Little-Endian Convention}

Consider the state space $\mathcal{N} = \{1, \dots, n_1\} \times \{1 , \dots, n_2\} \times \dots \times \{1, \dots, n_d\}$. The multi-index 
\begin{equation*}
  \overline{x_1, \dots, x_d} := \phi_\mathcal{N}(x_1, \dots, x_d),
\end{equation*}
for $X=(x_1, \dots, x_d)^T \in \mathcal{N}$, is defined by a bijection $\phi_\mathcal{N}$ with
\begin{equation*}
  \phi_\mathcal{N} : \mathcal{N} \rightarrow \{1, \ldots, \prod_{i=1}^{d} n_i\}, \qquad X \mapsto \phi_\mathcal{N}(X).
\end{equation*}
Using the \emph{little-endian} convention, this bijection is given by
\begin{equation*}
    \phi(x_1, \ldots, x_d) = 1+ (x_1 -1) + \ldots + (x_d -1) \cdot n_1 \cdot \ldots \cdot n_{d-1} = 1 + \sum_{i=1}^{d} (x_i -1) \prod_{j=1}^{i-1} n_j.
\end{equation*}

\section{Properties of the SLIM Decomposition}
\label{app: Properties of the SLIM Decomposition}

\subsection{Equivalence of SLIM  and Canonical Decomposition}

\begin{theorem}
 The SLIM decomposition given in \eqref{eq: SLIMTT cyclic, heterogeneous} corresponds to the canonical decomposition given in \eqref{eq: SLIMTT op 2}.
\end{theorem}
\begin{proof}
 Consider the first two TT cores of the SLIM decomposition, given by
 \begin{equation*}
  \begin{bmatrix}
        \mathbf{S}_1 & \mathbf{L}_1 & \mathbf{I}_1 & \mathbf{M}_1 
    \end{bmatrix}
    \otimes 
    \begin{bmatrix}
        \mathbf{I}_2 & 0            & 0            & 0            \\
        \mathbf{M}_2 & 0            & 0            & 0            \\       
        \mathbf{S}_2 & \mathbf{L}_2 & \mathbf{I}_2 & 0            \\
        0            & 0            & 0            & \mathbf{J}_2
    \end{bmatrix} =
    \begin{bmatrix}
        \mathbf{S}_1 \otimes \mathbf{I}_2 + \mathbf{I}_1 \otimes \mathbf{S}_2 + [\mathbf{L}_1 ] \otimes [\mathbf{M}_2]  & & \mathbf{I}_1 \otimes [\mathbf{L}_2] & & \mathbf{I}_1 \otimes \mathbf{I}_2 & & [\mathbf{M}_1] \otimes [\mathbf{J}_2] 
    \end{bmatrix} .
 \end{equation*}
  Successively, we obtain
  \begin{equation*}
  \begin{split}
    \mathbf{A} &= 
    \begin{bmatrix}
        \mathbf{S}_1 \otimes \mathbf{I}_2 \otimes \dots \otimes \mathbf{I}_{d-1} + \dots \\
				\dots + \mathbf{I}_1 \otimes \dots \otimes \mathbf{I}_{d-2} \otimes  \mathbf{S}_{d-1} \\+ [\mathbf{L}_1 ] \otimes [\mathbf{M}_2] \otimes \mathbf{I}_3 \otimes \dots \otimes \mathbf{I}_{d-1} + \dots \\
				\dots + \mathbf{I}_1 \otimes \dots \otimes \mathbf{I}_{d-3} \otimes [\mathbf{L}_{d-2} ] \otimes [\mathbf{M}_{d-1}] \\[0.25cm] \mathbf{I}_1 \otimes \dots \otimes \mathbf{I}_{d-2} \otimes [\mathbf{L}_{d-1}] \\[0.25cm] \mathbf{I}_1 \otimes \dots \otimes \mathbf{I}_{d-1} \\[0.25cm] [\mathbf{M}_1] \otimes [\mathbf{J}_2] \otimes \dots \otimes [\mathbf{J}_{d-1}]
    \end{bmatrix}^\mathbb{T}  \otimes 
    \begin{bmatrix}
        \mathbf{I}_{d} \\
        \mathbf{M}_{d} \\       
        \mathbf{S}_{d} \\
        \mathbf{L}_{d}
    \end{bmatrix}\\
    & = \mathbf{S}_1 \otimes \mathbf{I}_2 \otimes \dots \otimes \mathbf{I}_{d} + \dots + \mathbf{I}_1 \otimes \dots \otimes \mathbf{I}_{d-1} \otimes  \mathbf{S}_{d} \\
        & \quad + [\mathbf{L}_1 ] \otimes [\mathbf{M}_2] \otimes \mathbf{I}_3 \otimes \dots \otimes \mathbf{I}_{d} + \dots + \mathbf{I}_1 \otimes \dots \otimes \mathbf{I}_{d-2} \otimes [\mathbf{L}_{d-1} ] \otimes [\mathbf{M}_{d}]\\
        & \quad + \left[ \mathbf{M}_1\right] \otimes \left[ \mathbf{J}_2\right] \otimes \dots \otimes \left[ \mathbf{J}_{d-1} \right] \otimes \left[ \mathbf{L}_{d} \right], 
  \end{split}
  \end{equation*}
  which is exactly the same expression as \eqref{eq: SLIMTT op 2}.
\end{proof}

\subsection{Storage Consumption of the SLIM Decomposition}

\begin{theorem}
  The storage consumption in the sparse format of the SLIM decomposition for cyclic, heterogeneous NNISs as given in \eqref{eq: SLIMTT cyclic, heterogeneous} is
  \begin{equation*}
    O \left( \sum_{i=1}^d (\beta_{i-1} + \beta_i +1) n_i^2 + \sum_{i=2}^{d-1} (\beta_d +2)  n_i + n_1 + n_d \right),
  \end{equation*}
  with $\beta_0 = \beta_d$.
\end{theorem}
\begin{proof}
  We assume the matrices of the core elements $\mathbf{S}_i$, $\mathbf{L}_i$, and $\mathbf{M}_i$, $i = 1, \dots, d$, to be dense, i.e.~the storage consumption of a single matrix is then estimated as $O(n_i^2)$. For the components $\mathbf{I}_i$, we obtain $O(n_i)$ since $\mathbf{I}_i = I \in \R^{n_i \times n_i}$ has only $n_i$ entries. Furthermore, we can analogously estimate the storage of $\mathbf{J}_i$ as $O(\beta_d \cdot n_i)$. Thus, we have the following storage estimates for the different TT cores.
  \begin{equation*}
    \begin{array}{lcl}
      \mathbf{A}^{(1)} & : & O\left((\beta_d + \beta_1 +1)n_1^2 + n_1 \right), \\
      \mathbf{A}^{(i)}, 2 \leq i \leq d-1 & : & O\left((\beta_{i-1} + \beta_i +1)n_i^2 + (2+\beta_d)n_i \right), \\
      \mathbf{A}^{(d)} & : & O\left((\beta_{d-1} + \beta_d +1)n_d^2 + n_d \right).
    \end{array}
  \end{equation*}
  Summation over all cores concludes the proof.
\end{proof}

\end{document}